\documentclass{amsart}%
\usepackage{amsmath}
\usepackage{amsfonts}
\usepackage{amssymb}
\usepackage{graphicx}
\usepackage{bm}%
\newtheorem{theorem}{Theorem}[section]
\newtheorem{corollary}[theorem]{Corollary}
\newtheorem{lemma}[theorem]{Lemma}
\newtheorem{proposition}[theorem]{Proposition}
\theoremstyle{definition}

\newtheorem{example}[theorem]{Example}
\theoremstyle{remark}
\newtheorem{remark}[theorem]{Remark}
\numberwithin{equation}{section}

\newcommand{\supp}{\operatorname{supp}}

\begin{document}
\title[Scattering for Miura potentials]{Inverse scattering on the line for Schr\"{o}dinger operators with Miura
potentials, II. \\ Different Riccati representatives}
\author[Hryniv]{Rostyslav O. Hryniv}
\address[Hryniv]{Institute for Applied Problems of Mechanics and Mathematics, 3b
Naukova st., 79601 Lviv, Ukraine \and Department of Mechanics and Mathematics,
Lviv National University, 79602 Lviv, Ukraine \and Institute of Mathematics, the University of Rzesz\'{o}w, 16\,A Rejtana st., 35-959 Rzesz\'{o}w, Poland}
\email{rhryniv@iapmm.lviv.ua}
\author[Mykytyuk]{Yaroslav V. Mykytyuk}
\address[Mykytyuk]{Department of Mechanics and Mathematics, Lviv National University, 79602 Lviv, Ukraine}
\email{yamykytyuk@yahoo.com}
\author[Perry]{Peter A. Perry}
\address[Perry]{ Department of Mathematics, University of Kentucky, Lexington,
Kentucky, 40506-0027, U.S.A.}
\email{perry@ms.uky.edu}


\date{02 October 2009}

\begin{abstract}
This is the second in a series of papers on scattering theory for one-dimensional Schr\"{o}dinger operators with Miura potentials admitting
a Riccati representation of the form $q=u^{\prime}+u^{2}$ for some $u\in L^{2}(\mathbb{R})$. We consider potentials for which there exist `left' and
`right' Riccati representatives with prescribed integrability on half-lines.
This class includes all Faddeev--Marchenko potentials in $L^{1}\bigl({\mathbb{R}},(1+|x|)dx\bigr)$ generating positive Schr\"{o}dinger operators as well as many distributional potentials with Dirac delta-functions and Coulomb-like singularities. We completely describe the corresponding set of reflection coefficients~$r$ and justify the algorithm reconstructing~$q$ from~$r$.
\end{abstract}

\maketitle

\tableofcontents

\section{Introduction}\label{sec.intro}

This is the second in a series of
papers on inverse scattering for the Schr\"{o}dinger operators%
\[
S:=-\frac{d^{2}}{dx^{2}}+q(x)
\]
on the line with highly singular potentials $q$. Our eventual goal is to study
the~KdV equation with rough initial data using the method of inverse
scattering. In this paper, we will define a class of highly singular
potentials for which the direct scattering map is well-defined and invertible,
and obtain a complete characterization of the reflection coefficients that
arise. An important aspect of our work is the connection between inverse
scattering for the Schr\"{o}dinger equation and inverse scattering for the
ZS-AKNS system, obtained through the Riccati representation for singular
potentials that we explain below.

In a separate paper~\cite{HMP:2010}, we will use these results together with the Riemann--Hilbert formulation of the inverse scattering problem for the ZS-AKNS\ system (see
especially Zhou~\cite{Zhou:1998}) to obtain mapping properties of the
scattering transform between weighted Sobolev spaces of potentials and
corresponding weighted Sobolev spaces of reflection coefficients, in the
spirit of~\cite{Zhou:1998}. We will then use these fine mapping properties of
the scattering transform to study the KdV flow.

We consider real-valued potentials $q\in H^{-1}(\mathbb{R})$ having the property that the quadratic form%
\begin{equation}\label{eq.h}
    \mathfrak{s}(\varphi)
        = \int\left\vert \varphi'(x)\right\vert^{2}\,dx
        + \bigl\langle q,\left\vert \varphi\right\vert^{2}\bigr\rangle
\end{equation}
defined on $C_{0}^{\infty}(\mathbb{R})$ is nonnegative and obeying some
additional restrictions imposed in order to construct a meaningful scattering
theory. Here $\left\langle \,\cdot\,,\,\cdot\,\right\rangle ~$denotes the dual pairing of~$H^{-1}(\mathbb{R)}$ and~$H^{1}(\mathbb{R)}$.

As shown in~\cite{KPST:2005}, any potential of the above type admits a Riccati
representation of the form
\begin{equation}
q=u^{\prime}+u^{2}\label{eq:Miura}%
\end{equation}
for a real-valued function $u\in L^{2}(\mathbb{R})$. Such a potential~$q$ is
called a \emph{Miura potential}, and the nonlinear map defined by~\eqref{eq:Miura} is called the \emph{Miura map}~\cite{Miura:1968}. We refer
the reader to Appendix C of \cite{KPST:2005} for a discussion of related
literature on the Miura map and properties of positive solutions to the
Schr\"{o}dinger equation.

The Riccati representation of a Miura
potential is generally not unique. Any Riccati representative $u$ gives rise
to a strictly positive distributional solution~$y$ of the zero-energy
Schr\"{o}dinger equation $-y^{\prime\prime}+ qy =0$ via
\[
    y(x)=\exp\left(  \int_{0}^{x}u(s)~ds\right)
\]
and, conversely, any positive solution $y\in H_{\mathrm{loc}}^{1}(\mathbb{R})$
gives rise to a Riccati representative $u(x)=y^{\prime}(x)/y(x)$. Thus, the
set of Riccati representatives for a given distribution potential $q$ is
parameterized by normalized positive solutions to the zero-energy
Schr\"{o}dinger equation. The set of such solutions $y$, normalized so that
$y(0)=1$, is denoted $\operatorname*{Pos}(q)$. There are extremal solutions
$y_{\pm}$ in $\operatorname*{Pos}(q)$ with the properties that
\[
    \int_{0}^{\infty}\frac{ds}{y^{2}_{+}(s)}
        = \int_{-\infty}^{0}\frac{ds}{y^2_{-}(s)}
        = +\infty,
\]
and any $y\in\operatorname*{Pos}(q)$ takes the form~$y=\theta y_{+}%
+(1-\theta)y_{-}$ for some~$\theta\in[0,1]$. The corresponding extremal
Riccati representatives $u_{\pm}=\left( \log y_{\pm}\right)  ^{\prime}$ belong
to $L_{\mathrm{loc}}^{2}(\mathbb{R})$; we will assume in addition that
$u_{\pm}$ are in~$L^{2}(\mathbb{R})$ and that $u_{+}$ is integrable at
$+\infty$ and $u_{-}$ is integrable at $-\infty$. The set of all potentials
with the above properties is denoted by $\mathcal{Q}$, i.e.,
\[
    \mathcal{Q}
        := \{q =\overline{q} \in H^{-1}(\mathbb{R})\, :\,
        \exists u_{\pm}\in L^{2}(\mathbb{R})\cap L^{1}(\mathbb{R}^{\pm})
        \text{ s.t. }  q = u_{+}^{\prime}+ u_{+}^{2} = u_{-}^{\prime}+ u_{-}^{2}\}.
\]

The set~$\mathcal{Q}$ contains all real-valued potentials of
Faddeev--Marchenko class (i.e., potentials belonging to~$L^{1}(\mathbb{R}%
,(1+|x|)dx)$) generating non-negative Schr\"{o}\-din\-ger operators as well as
many singular potentials (e.g., with Dirac delta-functions and Coulomb-like
singularities), see Section~\ref{sec.riccati}. To every $q\in\mathcal{Q}$ there corresponds a well defined non-negative Schr\"{o}dinger operator, and our main aim is to show that the classic scattering theory known for Faddeev--Marchenko potentials can be extended to the whole set~$\mathcal{Q}$.

Within $\mathcal{Q}$ there is a dichotomy between \textquotedblleft
generic\textquotedblright\ potentials for which $y_{+}\neq y_{-}$ and
\textquotedblleft exceptional\textquotedblright\ potentials for which
$y_{+}=y_{-}$. \ This corresponds to the well-known dichotomy for regular
potentials (i.e., measurable real-valued functions~$q$ with
     $\int \bigl( 1+x^{2}\bigr) |q(x)|\,dx<\infty$)
between those $q$ for which the
reflection coefficient~$r$ satisfies $r(0)=-1$ (the \textquotedblleft
generic\textquotedblright\ case) and those for which $\left\vert
r(0)\right\vert <1$ (the ``exceptional'' case): see for example \cite{DT:1979}, section 2.3, Theorem 1 (pp.~146--147) and Remark 9 (pp.~152--153). Note that this dichotomy is invariant under the KdV flow on $\mathcal{Q} \cap S(\mathbb{R})$, since, under the KdV flow $t \mapsto q(\cdot,t)$, the reflection coefficient is given by $r(k,t)=\exp(8ik^3t)r_0(k)$.

If we write $v(x):=u_{-}(x)-u_{+}(x)$, it is easy to see that $v$ is actually a
continuous function, that $v$ is either identically zero or everywhere
nonvanishing, and that $v(0)>0$ for the \textquotedblleft
generic\textquotedblright\ potentials, while $v(0)=0$ for the exceptional ones
(see Section \ref{sec.riccati} and equation (\ref{eq.vx})). For this reason we
will denote the subset of \textquotedblleft generic\textquotedblright\ Miura
potentials by $\mathcal{Q}_{>0}$, and the subset of \textquotedblleft
exceptional\textquotedblright\ potentials by~$\mathcal{Q}_{0}$. The
corresponding sets of reflection coefficients (defined more precisely below)
will be denoted $\mathcal{R}_{>0}$ and~$\mathcal{R}_{0}$. A crucial
observation is that a potential $q\in\mathcal{Q}$ is uniquely characterized by
the data%
\begin{equation}
    \bigl(  \left.  u_{+}\right\vert _{(0,\infty)},\left.  u_{-}
        \right\vert_{(-\infty,0)},v(0) \bigr), \label{eq.riccati}%
\end{equation}
see Section~\ref{sec.riccati} and Lemmas \ref{lemma.r1} and \ref{lemma.r2}. If
we set $X_{0}^{+}=L^{1}(0,\infty)\cap L^{2}(0,\infty)$ and $X_{0}^{-}%
=L^{1}(-\infty,0)\cap L^{2}(-\infty,0)$, we can then topologize $\mathcal{Q}$
as $X_{0}^{+}\times X_{0}^{-}\times[0,\infty)$.

In the first paper~\cite{FHMP:2009} of this series (referred to as Paper~I in what follows), we studied the case of \textquotedblleft
exceptional\textquotedblright\ potentials and constructed the scattering and
inverse scattering maps as continuous bijections between $\mathcal{Q}_{0}$ and
$\mathcal{R}_{0}$. The goal of this paper is to study the generic case and
construct the direct and inverse scattering maps as continuous bijections
between $\mathcal{Q}_{>0}$ and $\mathcal{R}_{>0}$. The primary technical
challenge is to give a workable characterization of the low-energy asymptotic
behavior of the reflection and transmission coefficients for the class of
singular potentials under study (see Section~\ref{subsec.rep} below) and show
that this characterization is sufficient to prove that the standard formulas
from the Gelfand--Levitan--Marchenko theory carry through and give a correct
reconstruction (see Section~\ref{subsec.consistent} below). As a by-product,
we show that the reflection coefficients corresponding to $q\in\mathcal{Q}$
are continuous on the whole line. For real-valued potentials in~$L^{1}\bigl(\mathbb{R},(1+|x|)dx\bigr)$, Marchenko~\cite[Ch.~3.5]{Marchenko:1977} established this property generically and conjectured it for the exceptional case; then Deift and Trubowitz~\cite{DT:1979} proved the continuity of~$r$ for a
subset of potentials $q\in L^{1}(\mathbb{R},(1+x^{2})dx)$ (see also~\cite[Problem~3.5.3]{Marchenko:1977}), and finally Guseinov~\cite{Guseinov:1985} and independently Klaus~\cite{Klaus:1988} justified the above conjecture for exceptional Faddeev--Marchenko potentials.

In Paper~I, we used the correspondence between the Schr\"{o}dinger and
ZS-AKNS\ equations together with well-known inverse theory for the ZS-AKNS\ system
to obtain properties of Jost solutions, characterization of the transmission
and reflection coefficients, and a reconstruction algorithm. In the case
considered there, a single Riccati representative uniquely parameterizes the
potential and it suffices to study scattering for a single ZS-AKNS\ system.

In the present paper, we use the extremal Riccati representatives $u_{+}$ and
$u_{-}$ for a potential $q\in\mathcal{Q}$ to construct \textquotedblleft
right\textquotedblright\ and \textquotedblleft left\textquotedblright\ ZS-AKNS\ systems which yield the right and left Jost solutions, and
\textquotedblleft right\textquotedblright\ and \textquotedblleft
left\textquotedblright\ reconstruction formulas. Namely, for~$k \in\mathbb{C}$
with non-negative imaginary part, one can construct Jost solutions~$f_{\pm}(x,k)$ of the Schr\"{o}dinger equation
\[
    -y^{\prime\prime}+ q y = k^{2} y
\]
with%
\[
    \lim_{x\rightarrow\pm\infty}\left\vert f_{\pm}(x,k)-e^{\pm ikx}\right\vert =0
\]
(see \S \ref{subsec.jost} for detailed discussion). If $k$ is real and
nonzero, then the functions
\[
    f_{+}(x,k),\ f_{+}(x,-k)
\]
and
\[
    f_{-}(x,k),\ f_{-}(x,-k)
\]
are linearly independent solutions of the above Schr\"{o}dinger equation, so
that we may define coefficients $a(k)$ and $b(k)$ by the relation%
\begin{equation}\label{eq.ab+}%
    f_{+}(x,k)
        = a(k)f_{-}(x,-k)+b(k)f_{-}(x,k),
\end{equation}
or, equivalently, by
\begin{equation}\label{eq.ab-}%
    f_{-}(x,k)
        = a(k)f_{+}(x,-k)-b(-k)f_{+}(x,k).
\end{equation}
The associated reflection and transmission coefficients are, as usual, given
by%
\begin{align}
    r_{+}(k)  &  =-\frac{b(-k)}{a(k)},\label{eq.r+}\\
    r_{-}(k)  &  =\frac{b(k)}{a(k)},\label{eq.r-}\\
    t(k)  &  =\frac{1}{a(k)}. \label{eq.t}%
\end{align}

To characterize the reflection coefficients corresponding to potentials
in~$\mathcal{Q}$, we introduce the space $X:=L^{1}(\mathbb{R})\cap
L^{2}(\mathbb{R})$ with the norm $\Vert f\Vert_{X}=\Vert f\Vert_{L^{1}}+\Vert
f\Vert_{L^{2}}$ and denote by $\widehat{X}$ the set of Fourier
transforms~$\widehat{f}$ of functions in~$X$, with $\Vert\widehat{f}%
\Vert_{\widehat{X}}:=\Vert f\Vert_{X}$. Clearly, $\widehat{X}$ consists of
continuous functions. Introduce now the set
\[
    \mathcal{R}
        :=\{r\in\widehat{X}\,:\,r(-k)=\overline{r(k)},\ |r(k)|<1
        \text{ for } k\neq0\}
\]
and its subsets
\[
    \mathcal{R}_{0}:=\{r\in\mathcal{R}\,:\,|r(0)|<1\}
\]
and
\[
    \mathcal{R}_{>0}
        :=\{r\in\mathcal{R}\,:\,r(0)=-1,\ \widetilde{r}(k):=(1-|r(k)|^{2})/{k^{2}}\in\widehat{X}, \
        \widetilde r(0) >0\}.
\]
The topology in~$\mathcal{R}$ and $\mathcal{R}_0$ is inherited from that of~$\widehat X$.
The set~$\mathcal{R}_{>0}$ becomes a metric space with the metric defined by the distance
\[
    d(r_{1},r_{2}) := \left\Vert r_{1}-r_{2}\right\Vert _{\widehat{X}}
                    +\left\Vert \widetilde{r}_{1}-\widetilde{r}_{2}\right\Vert _{\widehat{X}}.
\]
Note that $\mathcal{R}_{>0}\cup\mathcal{R}_{0}$ is a proper subset of $\mathcal{R}$ and that $\mathcal{R}_{0}$ is open in~$\mathcal{R}$ while $\mathcal{R}_{>0}$ is neither open nor closed in~$\mathcal{R}$.

We denote by~$\mathcal{S}_+$ and~$\mathcal{S}_-$ the direct scattering maps that send a potential~$q$ in $\mathcal{Q}$ into the reflection coefficients $r_+$ and $r_-$, respectively. For the \textquotedblleft
exceptional\textquotedblright\ case studied in Paper~I, it was proved that $\mathcal{S}_\pm$ are homeomorphisms between $\mathcal{Q}_{0}$ and $\mathcal{R}_{0}$. Here we shall study the \textquotedblleft generic\textquotedblright%
\ case, and our main result is:

\begin{theorem}
\label{thm.main} The direct scattering maps $\mathcal{S}_\pm$ are homeomorphisms between~$\mathcal{Q}_{>0}$ and $\mathcal{R}_{>0}$.
\end{theorem}

We prove Theorem \ref{thm.main} in two steps. First, we construct  the direct
scattering maps $\mathcal{S}_\pm$ and study their properties (see Theorem~\ref{thm.direct}). We then construct
the inverse maps $\mathcal{S}_\pm^{-1}$ in Theorem \ref{thm.inverse} and give an
explicit reconstruction algorithm.

In Paper~III of this series~\cite{HMP:2009}, we will show how to add bound states to potentials $q \in \mathcal{Q}$ and thereby complete our analysis of the direct
and inverse scattering maps for singular potentials. In a separate paper~\cite{HMP:2010}, we study the direct and inverse scattering maps on subspaces
of $\mathcal{Q}$ with extremal Riccati representatives belonging to weighted Sobolev spaces, in the spirit of~Zhou~\cite{Zhou:1998}. This will allow us to give a complete characterization of potentials with reflection coefficients belonging to weighted Sobolev spaces $H^{j,k}(\mathbb{R})$, i.e., reflection coefficients with~$j$ distributional derivatives in~$L^2(\mathbb{R})$ and with~$s^k r(s) \in L^2(\mathbb{R})$. These spaces play an important role in the study of the~KdV equation and other equations in the~KdV hierarchy since, for example, the~KdV flow preserves reflection coefficients in the space $H^{1,2}(\mathbb{R})$. In~\cite{HMP:2010} we will also construct solutions of the~KdV equation with less regular initial data by the inverse scattering method.

The contents of this paper are as follows. In \S \ref{sec.riccati}, we review
the basic facts about Riccati representatives and the Riccati representation
(\ref{eq.riccati}) and give several examples of potentials in~$\mathcal{Q}$.
In \S \ref{sec.direct} we construct the Jost solutions, obtain representation
formulas for the transmission and reflection coefficients, and show that all
of the scattering data are determined by either the left or the right
reflection coefficient alone. In \S \ref{sec.inverse}, we give reconstruction
formulas for the Riccati representation (\ref{eq.riccati}) given a single
reflection coefficient, and prove consistency of the reconstruction formulas.

\emph{Notation.} In what follows, $\mathbb{R}^{+}=(0,\infty)$ and
$\mathbb{R}^{-}=(-\infty,0)$. Given a Miura potential $q\in H_{\mathrm{loc}%
}^{-1}(\mathbb{R})$, we shall always denote by $u_{\pm}$ the Riccati
representatives corresponding to extremal solutions~$y_{\pm}$ of the
equation~$-y^{\prime\prime}+qy=0$ and by $w_{\pm}$ the restrictions
of~$u_{\pm}$ to $\mathbb{R}^{\pm}$.

Next, $M_{2}(\mathbb{C})$ will stand for the linear space of $2\times2$
matrices with complex entries and $\left\vert\,\cdot\,\right\vert $ will denote
the Euclidean norm for vectors and matrices, while $\left\Vert\,\cdot\,\right\Vert$ will be used for norms on various function spaces.

We denote by $X$ the space $L^{1}(\mathbb{R})\cap L^{2}(\mathbb{R})$, by
$X^{+}$ the space $L^{2}(\mathbb{R})\cap L^{1}(\mathbb{R}^{+})$, and by $X^{-}$ the
space $L^{2}(\mathbb{R})\cap L^{1}(\mathbb{R}^{-})$. Thus $X^{\pm}$ are spaces of
functions on the real line with prescribed integrability at $\pm\infty$. For $c\in\mathbb{R}$, we also set
\begin{align*}
    X_{c}^{+}  &  :=L^{1}(c,\infty)\cap L^{2}(c,\infty),\\
    X_{c}^{-}  &  :=L^{1}(-\infty,c)\cap L^{2}(-\infty,c).
\end{align*}

It will be convenient to use a non-standard normalization of the Fourier transform~$\widehat f=\mathcal{F}f$
of a function $f\in L^1(\mathbb{R})$, viz.
\begin{equation}\label{eq.Fourier}
    \widehat{f}(\xi) = (\mathcal{F}f)(\xi):=\int_{-\infty}^\infty \exp(2i\xi x) f(x) \,dx.
\end{equation}
We will denote by $\widehat{X}$ the Banach algebra of functions whose Fourier transforms lie in $X$; $\widehat{X}$ is a subalgebra of the classical Wiener algebra. The unital extension of~$\widehat{X}$
is a Banach algebra consisting of all functions of the form~$c+\widehat{f}$
with $c\in\mathbb{C}$ and $f\in X$ and is denoted by~$1\dotplus\widehat{X}$.
An element $a$ of $1\dotplus\widehat{X}$ is invertible there if and only if
$a$ does not vanish on $\mathbb{R}$ and does not tend to zero at infinity
(cf.~Appendix of Paper~I).

Further, for every $r\in \mathcal{R}_{>0}$, we denote by $\widetilde{r}$ the element of~$\widehat X$ given by
\begin{equation}\label{eq.r.tilde.def}
    \widetilde{r}(k):= (1 - |r(k)|^2)/k^2 .
\end{equation}
It follows from the definition of the set~$\mathcal{R}_{>0}$ that the function
\[
    \bigl(1- |r(k)|^{2}\bigr)\frac{k^2+1}{k^2}
            = 1 - r(k)r(-k) + \widetilde{r}(k)
\]
does not vanish on the real line, belongs to $1\dotplus \widehat X$ and thus is an invertible element there.

Finally, $H_{+}^{2}(\mathbb{R})$ is the Hardy space of functions $F$ on the
upper half-plane with%
\[
    \left\Vert F\right\Vert _{H_{+}^{2}(\mathbb{R})}
        :=\sup_{y>0}\left\Vert F(\,\cdot+iy)\right\Vert _{L^{2}(\mathbb{R})}%
\]
finite. These functions are determined by their boundary values $f$ on
$\mathbb{R}$ (i.e., by the limits of~$F(\,\cdot+iy)$ as $y\to0+$ in the
topology of $L^{2}(\mathbb{R})$) and $\left\Vert F\right\Vert _{H_{+}%
^{2}(\mathbb{R})}=\left\Vert f\right\Vert _{L^{2}(\mathbb{R})}$.
The Hardy space $H_{-}^{2}(\mathbb{R})$ for the lower half-plane is defined analogously. Clearly, $H_{+}^{2}(\mathbb{R})$ (resp.~$H_{-}^{2}(\mathbb{R})$) consists of Fourier transforms of functions in $L^2(\mathbb{R})$ supported on the positive half-line~$\mathbb{R^+}$ (resp., on the negative half-line~$\mathbb{R^-}$), so that $L^{2}(\mathbb{R})=H_{-}^{2}(\mathbb{R})\oplus H_{+}^{2}(\mathbb{R})$. For $f\in L^{2}(\mathbb{R})$, we denote by ${\mathcal C}$ the Cauchy integral operator%
\[
    \left({\mathcal C}f\right)  (z)
        =\frac{1}{2\pi i}\int_{\mathbb{R}}\frac{1}{s-z}f(s)\,ds,
        \qquad z \in {\mathbb C} \setminus {\mathbb R},
\]
and by ${\mathcal C}_{\pm}$ the operators%
\begin{equation}\label{eq.Cauchy}%
    \left({\mathcal C}_{\pm}f\right)  (s)=\lim_{\varepsilon\downarrow0}
            \left(\mathcal{C}f\right)(s\pm i\varepsilon),
    \qquad s \in {\mathbb R},
\end{equation}
where the limit is taken in~$L^{2}(\mathbb{R})$. The operators~${\mathcal C}_{+}$ and~$-{\mathcal C}_{-}$ are Riesz orthogonal projections onto~$H_{+}^{2}(\mathbb{R})$ and~$H_{-}^{2}(\mathbb{R})$ respectively, and ${\mathcal C}_{+}-{\mathcal C}_{-}=I$ as operators on~$L^{2}(\mathbb{R})$. We have the formulas%
\[
    \left({\mathcal C}_{\pm}f\right)(s)
        = \pm\mathcal{F}^{-1}\chi_{\pm}\mathcal{F} f,
\]
where $\mathcal{F}^{-1}$ is the inverse of the Fourier transform~\eqref{eq.Fourier} and $\chi_{+}$ (resp. $\chi_{-}$) is the indicator function of~$\mathbb{R}^+$
(resp.\ of $\mathbb{R}^-$), implying that $\mathcal{C}_\pm$ are continuous operators in~$\widehat{X}$.

\section{Extremal solutions and Riccati representatives}

\label{sec.riccati}

Here we review some results from \cite{KPST:2005} (see especially Proposition
3.5 and Lemma 5.1 there) connecting positive zero-energy solutions and Riccati
representatives, and then justify the representation (\ref{eq.riccati}) for
potentials $q\in\mathcal{Q}$.

Suppose that $q\in H_{\mathrm{loc}}^{-1}(\mathbb{R})$ is a real-valued
distribution and define the quadratic form%
\[
\mathfrak{s}(\varphi)
    =\int\left\vert \varphi^{\prime}(x)\right\vert^{2}dx
      + \bigl\langle q,|\varphi|^{2}\bigr\rangle
\]
for $\varphi\in C_{0}^{\infty}(\mathbb{R})$, where $\left\langle
\,\cdot\,,\,\cdot\,\right\rangle $ denotes the dual pairing between
$H_{\mathrm{loc}}^{-1}(\mathbb{R})$ and $H_{\mathrm{comp}}^{1}(\mathbb{R})$.
We will denote by $\mathfrak{s}(\varphi,\psi)$ the associated sesquilinear
form. Set%
\[
    \lambda_{0}(q)
        =\inf\left\{ \mathfrak{s}(\varphi)\,:\,
        \varphi\in C_{0}^{\infty}(\mathbb{R}),\
        \left\Vert \varphi\right\Vert =1\right\}  .
\]
If $\lambda_{0}(q)\geq0$, then the space%
\[
    \operatorname*{Pos}(q)
        = \left\{y\in H_{\mathrm{loc}}^{1}(\mathbb{R})\,:\,
          y(0)=1,\ y>0,\ \mathfrak{s}(\varphi,y)=0
          \text{ for all }\varphi\in C_{0}^{\infty}(\mathbb{R})\right\}
\]
is nonempty and consists of normalized, positive distributional solutions to
the zero-energy Schr\"{o}\-din\-ger equation~$-y^{\prime\prime}+qy=0$. Given
any $y_{0}\in\operatorname*{Pos}(q)$, the function%
\[
    y_{1}(x)=y_{0}(x)\left(  1+c_1\int_{0}^{x}\frac{ds}{y^2_{0}(s)}\right)
\]
belongs to $\operatorname*{Pos}(q)$ whenever%
\[
    0\leq c_1\leq\left(  \int_{-\infty}^{0}\frac{ds}{y^2_{0}(s)}\right)^{-1},
\]
while%
\[
    y_{2}(x)=y_{0}(x)\left(  1+c_2\int_{x}^{0}\frac{ds}{y^2_{0}(s)}\right)
\]
belongs to $\operatorname*{Pos}(q)$ whenever
\[
    0\leq c_2\leq\left(\int_{0}^{\infty}\frac{ds}{y^2_{0}(s)}\right)^{-1}.
\]
There exist unique extremal elements $y_{\pm}$ of $\operatorname*{Pos}(q)$,
characterized respectively by
\[
    \int_{0}^{\infty}\frac{ds}{y^2_{+}(s)}=+\infty
\]
and%
\[
    \int_{-\infty}^{0}\frac{ds}{y^2_{-}(s)}=+\infty,
\]
so that any $y\in\operatorname*{Pos}(q)$ is written as $y=\theta
y_{+}+(1-\theta)y_{-}$ for a unique $\theta\in\left[  0,1\right]  $. If we set%
\begin{align*}
    m_{+} &  =\left(  \int_{-\infty}^{0}\frac{ds}{y^2_{+}(s)}\right)  ^{-1},\\
    m_{-} &  =\left(  \int_{0}^{ \infty}\frac{ds}{y^2_{-}(s)}\right)  ^{-1}%
\end{align*}
(with $m_{\pm}=0$ if the corresponding integral diverges), it is not difficult to show that%
\begin{align}
    y_{-}(x) &=y_{+}(x)\left(  1+m_{+}\int_{0}^{x}\frac{ds}{y^2_{+}(s)}\right)
        ,\label{eq.y-}\\
    y_{+}(x) &=y_{-}(x)\left(  1+m_{-}\int_{x}^{0}\frac{ds}{y^2_{-}(s)}\right)
        .\label{eq.y+}%
\end{align}
The ratio $y_{-}/y_{+}$ is continuously differentiable; computing its
logarithmic derivative at $x=0$ and using the
relations~\eqref{eq.y-}--\eqref{eq.y-} and $y_{\pm}(0)=1$, we find that%
\begin{equation}
    m_{+}={y_{-}^{\prime}(0)}-{y_{+}^{\prime}(0)}=m_{-}.\label{eq.m+-}%
\end{equation}
Henceforth we set $m_{+}=m_{-}=:m$. If $m=0$, then $y_{+}=y_{-}$, but otherwise the
extremal solutions are distinct. The following simple lemma is a direct
consequence of the observations above.

\begin{lemma}
The solutions $y_{\pm}$ are uniquely determined by the data%
\[
    \bigl( \left.  y_{+}\right\vert_{\mathbb{R}^+},
         \left.  y_{-}\right\vert_{\mathbb{R}^-}, m \bigr)  .
\]

\end{lemma}

The logarithmic derivatives
$
    u_{\pm}:={y_{\pm}^{\prime}}/{y_{\pm}}
$
belong to $L_{\mathrm{loc}}^{2}(\mathbb{R})$ and determine Riccati representatives for $q$. The function
\[
    v(x)=u_{-}(x)-u_{+}(x)
\]
satisfies $v^{\prime}=-(u_{+}+u_{-})v$ and thus is equal to%
\begin{equation}
    v(x) = v(0)\exp\Bigl\{ -\int_{0}^{x}
        \bigl[  u_{+}(s)+u_{-}(s)\bigr]\,ds\Bigr\}.\label{eq.vx}%
\end{equation}
Therefore $v$ is H\"{o}lder continuous of order~$\tfrac12$. If we recall that
$v(0)=m$, this shows that $v$ is either identically zero or strictly
positive, and suggests an alternative representation for $q$. Define $u\in
L_{\mathrm{loc}}^{2}(\mathbb{R})$ by%
\begin{equation}
    u(x)=\left\{ \begin{array} [c]{ll} u_{+}(x), & x>0,\\ u_{-}(x), & x<0.
                 \end{array}\right.  \label{eq.u}%
\end{equation}

\begin{lemma}\label{lemma.r1}
Suppose that $q\in H_{\mathrm{loc}}^{-1}(\mathbb{R})$ with
$\lambda_{0}(q)\geq0$, let $y_{+}$ and $y_{-}$ be the extremal positive
solutions for $q$, and let $u_{\pm}=y_{\pm}^{\prime}/y_{\pm}$. Also, define~$u$ by~(\ref{eq.u}) and set $v=u_{-}-u_{+}$. Then%
\[
    q=u^{\prime}+u^{2}+v(0)\delta_{0}%
\]
as distributions in $H_{\mathrm{loc}}^{-1}(\mathbb{R})$, where $\delta_{0}$ is
the Dirac $\delta$-distribution supported at $x=0$.
\end{lemma}

\begin{proof}
Let $q_{\ast}=u^{\prime}+u^{2}$. Then $q-q_{\ast}$ is a distribution in
$H_{\mathrm{loc}}^{-1}(\mathbb{R})$ with support at $x=0$, hence a tempered
distribution of the form $\alpha\delta_{0}$ by the regularity theorem for
tempered distributions. To evaluate $\alpha$ we test with a function of the
form $\varphi_{\varepsilon}(x)=\varphi(x/\varepsilon)$ where $\varphi\in
C_{0}^{\infty}(\mathbb{R})$ vanishes outside the interval~$[-1,1]$ and
satisfies~$\varphi(0)=1$. We compute (using $q=u_{+}^{\prime}+u_{+}^{2}$)%
\begin{align*}
    \alpha
        &=\left\langle q-q_{\ast},\varphi_{\varepsilon}\right\rangle \\
        &=\int_{-\infty}^{0}v(x)\frac{d}{dx}\varphi_{\varepsilon}(x)\,dx
            +\int_{-\infty}^{0}\left[u_{+}^{2}(x)-u_{-}^{2}(x)\right]  \varphi_{\varepsilon}(x)\,dx
\end{align*}
The second right-hand term vanishes as $\varepsilon\downarrow0$ since $u_{+}$
and $u_{-}$ belong to $L_{\mathrm{loc}}^{2}(\mathbb{R})$. In the first term we
have%
\[
    \int_{-\infty}^{0}v(x)\frac{d}{dx}\varphi_{\varepsilon}(x)\,dx
        = \int_{-\infty}^{0}v(0)\frac{d}{dx}\varphi_{\varepsilon}(x)\,dx
            +\int_{-\infty}^{0}\left[ v(x)-v(0)\right]
                \frac{d}{dx}\varphi_{\varepsilon}(x)\,dx.
\]
The first right-hand term gives $v(0)$. Using H\"{o}lder continuity of $v$ we
estimate the second right-hand term as follows:
\[
    \biggl|\int_{-\infty}^{0}\left[ v(x)-v(0)\right]
            \frac{d}{dx}\varphi_{\varepsilon}(x)\,dx\biggr|
    \le C_{1} \varepsilon^{-1}\int_{-\varepsilon}^{0}
        \left\vert x\right\vert^{1/2}\,dx
        = C_{2}\varepsilon^{1/2},
\]
with some positive constants~$C_1$ and $C_2$.
Thus, as $\varepsilon\downarrow0$, the second right-hand term vanishes, and
$\alpha=v(0)$.
\end{proof}

Lemma~\ref{lemma.r1} says that a real-valued distribution $q\in
H_{\mathrm{loc}}^{-1}(\mathbb{R})$ with $\lambda_{0}(q)\geq0$ is uniquely
determined by the data
\[
    \bigl( \left.  u_{+}\right\vert_{\mathbb{R}^+},
         \left.  u_{-}\right\vert_{\mathbb{R}^-}, v(0)\bigr).
\]
It turns out that the restrictions $\left. u_{+}\right\vert _{\mathbb{R}^+}$ and
$\left. u_{-}\right\vert _{\mathbb{R}^-}$ of the extremal Riccati
representatives and the value $v(0)$ are independent coordinates
in~$\mathcal{Q}$. Namely, the following statement holds true.

\begin{lemma}
\label{lemma.r2} Assume $w_{+}$ and $w_{-}$ are real-valued functions in
$X_{0}^{+}$ and $X_{0}^{-}$ respectively and that $\alpha\ge0$. Then there
exists a unique distribution $q\in\mathcal{Q}$ whose extremal Riccati
representatives $u_{\pm}$ and $v=u_{-}-u_{+}$ satisfy
\[
 \bigl( \left.  u_{+}\right\vert_{\mathbb{R}^+},
         \left.  u_{-}\right\vert_{\mathbb{R}^-}, v(0) \bigr)
            =\left(  w_{+},w_{-},\alpha\right).
\]
\end{lemma}

\begin{proof}
Uniqueness of~$q$ follows from Lemma~\ref{lemma.r1}. To prove existence, we
show how to construct $y_{+}$ and $y_{-}$ which are zero-energy extremal
positive solutions associated to a single distribution~$q$ looked for. We set%
\[
    y_{+}(x)=\exp\left(  \int_{0}^{x}w_{+}(s)\,ds\right)
\]
for $x>0$,
\[
    y_{-}(x)=\exp\left(  \int_{0}^{x}w_{-}(s)\,ds\right)
\]
for $x<0$, extend $y_{+}$ to $\mathbb{R}^-$ using (\ref{eq.y+}) with
$m_{-}=\alpha$, and extend $y_{-}$ to $\mathbb{R}^+$ using (\ref{eq.y-}) with
$m_{+}=\alpha$. We then set $u_{+}:=y_{+}^{\prime}/y_{+}$ and $u_{-}%
:=y_{-}^{\prime}/y_{-}$, and compute that
\[
    u_{+}(x) =  \begin{cases}
        w_{-}(x) - \alpha y_{-}^{-2}(x) \left/  \Bigl(1 + \alpha\int_{x}^{0}
                y_{-}^{-2}(s)\,ds\Bigr) \right. , \qquad & x<0,\\
        w_{+}(x), & x>0,
                \end{cases}
\]
with a similar expression for $u_{-}$. In particular, it follows that
$v:=u_{-}-u_{+}$ is a continuous function with~$v(0)=\alpha$. Since $w_{\pm}$
belong to $X_{0}^{\pm}$, we see that $u_{+}\in X_{0}^{+}$ and that $y_{-}$
tends to a nonzero limit as $x\to-\infty$. Thus $u_{+}$ belongs also to
$L^{2}(\mathbb{R}^{-})$, i.e., $u_{+}\in X^{+}$. Similar arguments show that
$u_{-}\in X^{-}$.

Set $q_{\pm}:=u_{\pm}^{\prime}+u_{\pm}^{2}$; then straightforward
calculations show that the distributions $q_+$ and $q_-$ coincide outside the origin; namely, in the distributional sense, they are equal to~$w_{-}^{\prime} + w_{-}^{2}$ for $x<0$ and to $w_{+}^{\prime} + w_{+}^{2}$ for $x>0$. The difference $q_{+}-q_{-}$ is a distribution belonging to~$H^{-1}(\mathbb{R})$ and supported at $x=0$, and thus must be of the form $\beta\delta_{0}$. Taking $\varphi_{\varepsilon}$ as in the proof of Lemma~\ref{lemma.r1} we compute%
\[
    \beta
        = \lim_{\varepsilon\downarrow0}
         \left(\int v(x)\frac{d}{dx}\varphi_{\varepsilon}(x)\,dx
        + \int\varphi_{\varepsilon}(x)\left[u^2_{+}(x)-u^2_{-}(x)\right]\,dx\right)  .
\]
The second right-hand term converges to $0$ as $\varepsilon\downarrow0$ and
the first right-hand term is given by
\[
    \int v(0)\frac{d}{dx}\varphi_{\varepsilon}(x)\,dx
     + \int\left[  v(x)-v(0)\right] \frac{d}{dx}\varphi_{\varepsilon}(x)\,dx.
\]
The first term here is zero and the second one can be estimated as before to
prove that $\beta=0$, as claimed.

This shows that $u_{+}$ and $u_{-}$ represent the same distribution $q$, which
thus belongs to~$\mathcal{Q}$. By construction, $y_{\pm}$ are extremal
positive solutions of the equation $-y^{\prime\prime}+qy=0$, and $u_{\pm}$ are
the corresponding extremal Riccati representatives with prescribed
restrictions $w_{\pm}=\left.  u_{\pm}\right\vert _{\mathbb{R}^{\pm}}$ and the
value of $v(0)$. The proof is complete.
\end{proof}

In what follows we will use the Riccati representation for elements of
$\mathcal{Q}_{>0}$, and we will topologize $\mathcal{Q}_{>0}$ by the topology
of $X_{0}^{+}\times X_{0}^{-}\times\mathbb{R}^{+}$ on Riccati representations.
In Paper I, we topologized $\mathcal{Q}_{0}$ by $X$ since $v(0)=0$ and
$u_{+}=u_{-}$. One should think of those $q\in\mathcal{Q}$ with $v(0)=0$ as
lying \textquotedblleft at infinity\textquotedblright\ in the topology of
$\mathcal{Q}_{>0}$.

If $q\in\mathcal{Q}$ corresponds to a triple $(w_+,w_-,v(0))\in X^+\times X^- \times [0,\infty)$, then the quadratic form~$\mathfrak{s}$ is closed on~$H^1(\mathbb{R})$~\cite{HM:2001} and the corresponding Schr\"{o}dinger operator~$S$ can be written as~(cf.~\cite{KPST:2005})
\[
    S = \Bigl(\frac{d}{dx}-w\Bigr)^*\Bigl(\frac{d}{dx}-w\Bigr) + v(0) \delta_0,
\]
where $w$ is a function in $X$ whose restrictions onto $\mathbb{R}^+$ and $\mathbb{R}^-$ coincide with $w_+$ and $w_-$ respectively. Therefore $S$ is indeed a non-negative operator.

We shall need the following property of non-negative Schr\"{o}dinger operators.

\begin{lemma}\label{lemma.jost}
Assume that $q\in H^{-1}(\mathbb R)$ is such that $\lambda_0(q) \geq 0$ and the Jost solutions $f_\pm(\cdot,0)$ exist. Then $f_\pm(\,\cdot\,,0)$ are strictly positive on~$\mathbb{R}$.
\end{lemma}

\begin{proof}
We discuss only the right Jost solution $f_+(\,\cdot\,,0)$, as the proof for $f_-(\,\cdot\,,0)$ is completely analogous.

Since $f_+(x,0)$ tends to~$1$ at $+\infty$, it remains within the interval~$(\tfrac12,2)$ for all $x$ greater than some $x_0\in\mathbb{R}$. The function
\[
   g(x) = f_+(x,0) \int_{x_0}^x \frac{ds}{f^2_+(x,0)}
\]
is a solution of the equation $-y'' + qy=0$ on~$(x_0,\infty)$ that is linearly independent of~$f_+(\cdot,0)$ and obeys there the bound $g(x) \ge (x-x_0)/8$. It follows that any element of $\operatorname*{Pos}(q)$ is either a multiple of $f_+(\cdot,0)$ or grows at infinity at least linearly. Since the extremal solution $y_+$ cannot have such a growth at infinity, we conclude that it must be a multiple of the Jost solution~$f_+(\cdot,0)$, and thus the latter is positive everywhere on~$\mathbb{R}$.
\end{proof}

\begin{remark}
\label{rem.top} Let as usual $u_{\pm}$ be the extremal Riccati representatives
for a $q\in\mathcal{Q}$, and let $w_{\pm}$ be their respective restrictions
to~$\mathbb{R}^{\pm}$. To show that a quantity depends continuously on
$\bigl(  w_{+},w_{-},v(0)\bigr)$, it will suffice to show that the same
quantity depends continuously on the restrictions of $u_{+}$ to a half-line
$(a,\infty)$, and of $u_{-}$ to a half-line $(-\infty,b)$ in the topology of
$X_{a}^{+}\times X_{b}^{-}$ for some $a<0$ and $b>0$. The reason is that these
restrictions determine $v(0)$ uniquely and continuously.
\end{remark}

We conclude this section with several examples of potentials in $\mathcal{Q}$.
The first shows that $\mathcal{Q}$ contains all Faddeev--Marchenko potentials
generating non-negative Schr\"{o}dinger operator~$S$ and the other demonstrate
that potentials in~$\mathcal{Q}$ might have local singularities typical for
$H_{\mathrm{loc}}^{-1}$---e.g., of Dirac delta-function type or Coulomb type.

\begin{example}\label{ex.FM}
Assume that $q\in L^{1}\bigl(\mathbb{R},(1+|x|)dx\bigr)$ is a
real-valued function of Faddeev--Marchenko class for which the corresponding
Schr\"{o}dinger operator~$S$ is non-negative. Denote by $f_{+}(\,\cdot\,,0)$ and
$f_{-}(\,\cdot\,,0)$ the Jost zero-energy solutions for~$S$; then $f_{\pm}(x,0)\to1$
as $x\to\pm\infty$. Since $S$ is non-negative and $q\in H^{-1}(\mathbb{R})$, the functions $f_{\pm}(\cdot,0)$ do not vanish on~$\mathbb{R}$ by Lemma~\ref{lemma.jost}.

It follows that~$f_{+}(\cdot,0)$ is bounded from above and bounded away from zero on the half-lines $(c,\infty)$ for every $c\in\mathbb{R}$. The function
 $y_{+} = f_{+}(\cdot, 0)/f_{+}(0,0)$
is the positive extremal solution corresponding to the potential~$q$, i.e.,
\[
    - y_{+}^{\prime\prime}(x) + q(x)y_{+}(x) =0.
\]
Since $f^{\prime}_{+}(x,0)$ tends to zero as $x\to\infty$ by the classical
theory, we conclude that
\[
    y_{+}^{\prime}(x) = - \int_{x}^{\infty}q(t) y_{+}(t)\,dt
\]
is bounded on every half-line $(c,\infty)$. Therefore,
\begin{align*}
\int_{0}^{\infty} |y_{+}^{\prime}(x)|\,dx
    & \le\sup_{x\ge0}|y_{+}(x)|\int_{0}^{\infty}\int_{x}^{\infty}|q(t)|\,dt\,dx\\
    & \le\sup_{x\ge0}|y_{+}(x)| \int_{0}^{\infty}t|q(t)|\,dt < \infty,
\end{align*}
so that $y_{+}^{\prime}$ belongs to $L^{1}(\mathbb{R}^{+})$ and, in view of
its boundedness, to~$L^{2}(\mathbb{R}^{+})$.

We now conclude that the corresponding ``right'' extremal Riccati
representative $u_{+}:=y_{+}^{\prime}/y_{+}$ belongs to~$X_{0}^{+}$ on the
positive half-line. Analogous arguments show that the restriction of similarly
constructed ``left'' extremal Riccati representative~$u_{-}$ onto~$\mathbb{R}%
^{-}$ belongs to~$X_{0}^{-}$. Lemmas~\ref{lemma.r1} and \ref{lemma.r2} now
imply that $q\in\mathcal{Q}$.
\end{example}

\begin{example}[{\cite[Example~1.6]{FHMP:2009}}]
Let $u$ be an even function that for $x>0$
equals $x^{-\alpha}\sin x^{\beta}$. Assume that $\alpha>1$ and $\beta
>\alpha+1$. Then $u$ belongs to~$X$ and the corresponding Miura potential
$q=u^{\prime}+u^{2}$ is of the form
\[
    q(x)= \beta\,\mathrm{sign}\,(x) |x|^{\beta-\alpha-1}\cos|x|^{\beta}
        +\tilde{q}(x)
\]
for some bounded function $\tilde{q}$. Thus $q$ is unbounded and oscillatory
but nevertheless belongs to~$\mathcal{Q}$.
\end{example}

\begin{example}\label{ex.delta} Set $q=\alpha\delta_{0}$, with~$\alpha>0$ and with~$\delta_{0}$ denoting the Dirac delta-function supported at~$x=0$. Recalling the definition of the corresponding Schr\"{o}\-din\-ger operator~\cite{AGHH:2005},
we conclude that the extremal solutions $y_{\pm}$ are different and equal
\[
    y_{+}(x)= \begin{cases} 1, & \ x>0\\ 1-\alpha x, & \ x<0 \end{cases}
            \qquad\text{and} \qquad
    y_{-}(x)= \begin{cases} 1+\alpha x, & \ x>0\\ 1, & \ x<0 \end{cases}
\]
respectively (see also~\cite{Levitan:1979}, \cite[Appendix~A]{KPST:2005},
and~\cite[Example~1.6]{FHMP:2009}). The corresponding extremal Riccati
representations are found to be
\[
    u_{+}(x)= \begin{cases} 0, & \ x>0\\
            \dfrac{-\alpha}{1-\alpha x}, & \ x<0 \end{cases}
            \qquad\text{and} \qquad
    u_{-}(x)= \begin{cases} \dfrac\alpha{1+\alpha x}, & \ x>0\\
            0, & \ x<0 \end{cases}
\]
respectively. Clearly, we have $u_{\pm}\in X^{\pm}$, so that $q\in\mathcal{Q}$.
\end{example}

\begin{example}[{\cite[Example~1.5]{FHMP:2009}}]\label{ex.coulomb}
Assume that $\phi\in C_0^\infty(\mathbb{R})$ is such that $\phi\equiv1$ on $(-1,1)$. Take $u(x) = \alpha\,\phi(x)\log|x|$ with $\alpha\in\mathbb{R}\setminus\{0\}$. Then~$u\in X$; moreover, since the distributional derivative of $\log|x|$ is the distribution~${\mathrm{P.v.}}\,1/x$, the corresponding Miura potential~$q=u'+u^2$ is smooth outside the origin and has there a Coulomb-type singularity. See e.g.~\cite{BDL:2000,Kurasov:1996,FLM:1995} and the references therein for discussion and rigorous treatment of Schr\"odinger operators with Coulomb potentials.
\end{example}

\begin{example}
\label{ex.coulomb-even} Potentials in~$\mathcal{Q}$ might have local
singularities typical for~$H^{-1}(\mathbb{R})$. Namely, take any real-valued distribution~$q\in H^{-1}(\mathbb{R})$ of compact support, for which the corresponding Schr\"odinger operator~$S$ is non-negative. We then claim that $q \in \mathcal{Q}$.

Indeed, any~$q \in H^{-1}(\mathbb{R})$ of compact support can be represented as~$Q'$ for a real-valued~$Q\in L^2_{\textrm{loc}}(\mathbb{R})$ that vanishes to the right of~$\supp q$ and is constant to the left of~$\supp q$.
The equation $-y^{\prime\prime}+qy =0$ should now be interpreted as the first-order system
\[
\frac{d}{dx}\binom{y_{1}}{y_{2}}
    = \begin{pmatrix} Q & 1\\ -Q^{2} & -Q \end{pmatrix}
            \binom{y_{1}}{y_{2}}
\]
for $y_1:=y$ and $y_2:= y'-Qy$.
Since $Q$ belongs locally to $L^{2}(\mathbb{R})$, for every $x_0\in\mathbb{R}$ and every complex numbers~$c_{1}$ and~$c_{2}$ the above system has a unique global absolutely continuous solution assuming the
prescribed values $y_{1}(x_0)=c_{1}$ and $y_{2}(x_0)=c_{2}$ at the point~$x=x_0$.

In particular, the solution of $-y^{\prime\prime}+qy=0$ that is
identically~$1$ to the right of~$\supp q$ admits a unique continuation to the whole line thus
giving the Jost solution~$f_{+}(\,\cdot\,,0)$. Since the operator~$S$ is
non-negative,  by Lemma~\ref{lemma.jost} $f_{+}(\,\cdot\,,0)$ never vanishes on the real line. Moreover,
$f_{+}(\,\cdot\,,0)$ is absolutely continuous, $f_{+}^{\prime}(\,\cdot\,,0)=y_{2}+Qy_{1}$ belongs locally to~$L^{2}(\mathbb{R})$, and thus the corresponding extremal Riccati representative~$u_{+}:=f_{+}^{\prime}(\,\cdot\,,0)/f_{+}(\,\cdot\,,0)$ belongs to~$L^{2}_{\mathrm{loc}}(\mathbb{R})$ as well. Since, moreover, $u_{+}=0$ to the right of~$\supp q$, we conclude that the restriction of~$u_{+}$ onto~$\mathbb{R}^{+}$ belongs to~$X_{0}^{+}$.

Similar constructions and arguments give the ``left'' extremal Riccati
representative~$u_{-}$ of~$q$ whose restriction onto~$\mathbb{R}^{-}$ belongs
to~$X_{0}^{-}$. By Lemmas~\ref{lemma.r1} and \ref{lemma.r2} we conclude that
$q\in\mathcal{Q}$.
\end{example}

The above reasoning can be adapted to a more general situation as follows.

\begin{proposition}
\label{pro.perturb} Assume that $q_{0}\in\mathcal{Q}$ and that $q_{1}$ is a
real-valued distribution in~$H^{-1}(\mathbb{R})$ with compact support such
that the Schr\"{o}dinger operator
\[
    -\frac{d^{2}}{dx^{2}} + q_{0} + q_{1}
\]
is nonnegative. Then $q_{0}+q_{1} \in\mathcal{Q}$.
\end{proposition}

\begin{proof}
As explained in Example~\ref{ex.coulomb-even}, for
 $q\in H^{-1}_{\mathrm{loc}}(\mathbb{R})$
every local distributional solution of the equation $-y^{\prime\prime}+qy=0$ can be continued to the whole line and belongs to~$H^{1}_{\mathrm{loc}}(\mathbb{R})$.

Since the perturbation~$q_{1}$ is of compact support, the ``right'' Jost
solution at zero energy for the Schr\"odinger operator with potential~$q_{0}%
+q_{1}$ coincides with that for the unperturbed operator (i.e., with
potential~$q_{0}$) to the right of the support of~$q_{1}$. Therefore the
``right'' extremal Riccati representatives for $q_{0}$ and $q_{0}+q_{1}$
(equal to the logarithmic derivatives of these Jost solutions) coincide for
large~$x$. Since they belong locally to $L^{2}$, we conclude that the
restriction of the ``right'' extremal Riccati representative for $q_{0}+q_{1}$
to~$\mathbb{R}^{+}$ belongs to $X_{0}^{+}$.

Similarly, the restriction of the
``left'' extremal Riccati representative for $q_{0}+q_{1}$ to~$\mathbb{R}^{-}$
belongs to~$X_{0}^{-}$, and thus $q_{0}+q_{1} \in\mathcal{Q}$ by
Lemmas~\ref{lemma.r1} and~\ref{lemma.r2}.
\end{proof}

For instance, assume that $\alpha_{j}$ and $\beta_{j}$, $j=1,\dots,n$, are
non-negative numbers and that $x_{1}<x_{2}<\dots<x_{n}$. Take $q\in\mathcal{Q}$ and an arbitrary real-valued $\phi \in C_0^\infty(\mathbb{R})$ and set $u_j(x):= \phi(x-x_j)\log|x-x_j|$; then the potential
\[
    q + \sum_{j=1}^{n} \alpha_{j}\delta_{x_{j}}
        + \sum_{j=1}^{n} \beta_{j} (u_j' + u_j^2)
\]
also belongs to~$\mathcal{Q}$. Note that any such potential has a much simpler representation as $w'+w^2+\alpha \delta_0$ for some real-valued $w\in X$ and $\alpha\ge0$.


\section{Direct scattering}\label{sec.direct}

In this section, we prove:

\begin{theorem}
\label{thm.direct} Suppose that $q\in\mathcal{Q}_{>0}$ with Riccati
representation $\left(  w_{+},w_{-},v(0)\right)  $. Then the reflection
coefficients $r_{\pm}$ are well defined and belong to $\mathcal{R}_{>0}$.
Moreover, the maps%
\begin{equation}\label{eq.Spm}%
    \mathcal{S}_{\pm}\,:\,
        (w_{+},w_{-},v(0))  \mapsto r_{\pm}
\end{equation}
and
\begin{equation}\label{eq.tmap}%
        (w_{+},w_{-},v(0))  \mapsto \widetilde{r}_\pm,
\end{equation}
with $\widetilde{r}_{\pm}$ as in~\eqref{eq.r.tilde.def},
are continuous from~$X_{0}^{+}\times X_{0}^{-}\times\mathbb{R}^+$ to~$\widehat{X}$. The maps $\mathcal{S}_{\pm}$ are one-to-one.
\end{theorem}

In \S \ref{subsec.jost} we review the construction of Jost solutions $f_{\pm}(\,\cdot\,,k)$ for potentials $q\in\mathcal{Q}$ and of the coefficients $a(k)$ and $b(k)$ that relate the Jost solutions $f_{\pm}(\,\cdot\,,k)$. We use the representation formulas for~$f_{\pm}(\,\cdot\,,k)$ in~\S \ref{subsec.rep} to construct and characterize the standard transmission
and reflection coefficients. We end \S \ref{subsec.rep} with a version of the
classical Levinson's theorem which states that a single reflection coefficient
determines $q\in\mathcal{Q}$ uniquely. In \S \ref{subsec.recon}, we show how
$t(k)$ may be obtained from either of $r_{\pm}$ and construct a continuous
involution~$\mathcal{I}$ on the space of reflection coefficients with $r_{\pm
}=\mathcal{I}r_{\mp}\,$.

\subsection{Jost solutions}\label{subsec.jost}

Suppose that $q\in H_{\mathrm{loc}}^{-1}(\mathbb{R})$. We say that $y\in
H_{\mathrm{loc}}^{1}(\mathbb{R})$ is a solution of the Schr\"{o}dinger
equation $-y^{\prime\prime}+qy=k^{2}y$ if
\[
    \mathfrak{s}(y,\varphi)=k^{2}\int_\mathbb{R} y \,\overline{\varphi}\,dx
\]
for all $\varphi\in C_{0}^{\infty}(\mathbb{R})$. It will be important to have
an alternative formulation.

\begin{proposition}
\label{prop.se.to.fo}Suppose that $q\in H_{\mathrm{loc}}^{-1}(\mathbb{R})$ possesses a Riccati representative $u\in L_{\mathrm{loc}}^{2}(\mathbb{R})$ and that $k\in\mathbb{C}$. A function $y\in H_{\mathrm{loc}}^{1}(\mathbb{R})$ solves the equation
\begin{equation}\label{eq.k2}
    -y^{\prime\prime}+qy=k^{2}y
\end{equation}
if and only if the vector-valued function
\[
    \binom{y}{y^{[1]}} :=\binom{y}{y^{\prime}-uy}
\]
solves the first-order system%
\begin{equation}\label{eq.fo}%
    \frac{d}{dx}\left(\begin{array} [c]{c} y\\ y^{[1]} \end{array} \right)
        =\left(\begin{array} [c]{ll} u & 1\\ -k^{2} & -u \end{array} \right)
        \left( \begin{array} [c]{c} y\\ y^{[1]} \end{array} \right) .
\end{equation}
\end{proposition}

The proof can be found e.g.\ in~\cite{SS:2003}. The function $y^{\left[
1\right]  }$ is called the \emph{quasi-derivative} of $y$. Note that $y^{[1]}$
is absolutely continuous, hence differentiable almost everywhere, while
$y^{\prime}$ need not even be continuous. Associated to this first-order
system is the modified Wronskian%
\begin{equation}\label{eq.W}%
    W\left\{  y_{1},y_{2}\right\}
        =y_{1}y_{2}^{\left[  1\right]  }-y_{2} y_{1}^{\left[  1\right]  }
\end{equation}
which is independent of $x$ if $y_{1}$ and $y_{2}$ are solutions
of~\eqref{eq.k2} and coincides with the ordinary Wronskian if $u$ is
absolutely continuous.

Before describing the construction of Jost solutions, we first recall the ZS-AKNS\ system and its well-known connection to the Schr\"{o}dinger equation. The ZS-AKNS\ system\footnote{This is actually a special case since we assume that the off-diagonal entries of $Q$ are equal and
real-valued. We also make a choice of the spectral parameter $z$ which is
convenient for connecting with the Schr\"{o}dinger equation.} is given by%
\begin{equation}\label{eq.AKNS}%
    \frac{d}{dx}\bm{\psi}=iz\sigma_{3}\bm{\psi}+Q\bm{\psi}
\end{equation}
where%
\begin{equation}\label{eq.sigmaQ}
    \sigma_{3}=\left( \begin{array} [c]{cc} 1 & 0\\ 0 & -1 \end{array} \right), \qquad
    Q(x)=\left(\begin{array} [c]{cc} 0 & u(x)\\ u(x) & 0 \end{array} \right).
\end{equation}
If $\bm{\psi}=\left(  \psi_{1},\psi_{2}\right)  ^{T}$ is a vector-valued solution
to (\ref{eq.AKNS}), then the scalar function%
\[
    \chi=\psi_{1}+\psi_{2}%
\]
solves the Schr\"{o}dinger equation $-\chi^{\prime\prime}+q\chi=z^{2}\chi$
with potential $q=u^{\prime}+u^{2}$. A short computation shows that the
quasi-derivative of $\chi$ is given by
\[
    \chi^{[1]}:=\chi^{\prime}-u\chi=iz(\psi_{1}-\psi_{2}).
\]

We can use this connection between ZS-AKNS systems and Schr\"{o}dinger equations, together with the representation formulas for solutions of (\ref{eq.AKNS}) derived in~\S3 of Paper I, to find representation formulas for the Jost solutions $f_{\pm}$ associated to a Miura potential $q$ and for their quasi-derivatives. Although in Paper~I it was assumed that $u_{+}=u_{-}=:u$ belongs to $L^{1}\cap L^{2}$ on the whole line, only the integrability of~$u$ on half-lines was used to construct the Jost solutions and to study their properties. Therefore it is legitimate to use the
respective results of Paper~I in the more general setting of this paper.

Assume now that $q\in\mathcal{Q}$, and let as usual $u_\pm$ be the corresponding extremal Riccati representatives belonging to $X^\pm$. Consider first the construction of the ``right'' Jost solution $f_{+}$. By the results of Paper~I the AKNS system~(\ref{eq.AKNS}) with $u=u_{+}$ and $z=k\in\mathbb{R}$ possesses a matrix-valued solution~$\Psi_{+}(\,\cdot\,,k)$ obeying the asymptotic condition%
\[
    \lim_{x\rightarrow+\infty}
        \left\vert \Psi_{+}(x,k)-\exp(ikx\sigma_{3})\right\vert =0.
\]
We then have the representation (taking $z=2k$ in the formulas of Paper I)%
\begin{equation}\label{eq.Psi+.rep}%
    \Psi_{+}(x,k)
        =\Bigl[  I + \int_{0}^{\infty}\Gamma_{+}(x,\zeta)
            \exp(2ik\zeta\sigma_{3})\,d\zeta\Bigr]  \exp(ikx\sigma_{3}),
\end{equation}
with $I$ being the $2\times2$ identity matrix and $\Gamma_{+}=(\Gamma_{jk}^{+})$ a matrix-valued function. Denoting by $X_{0}^{+}\otimes M_{2}(\mathbb{C})$ the space of~$2\times2$ matrix-valued functions on~$\mathbb{R}^{+}$ with entries
in~$X_{0}^{+}$, we get:

\begin{lemma}\label{lemma.Gamma+}
The map $x\mapsto\Gamma_{+}(x,\,\cdot\,)$ is continuous
from $\mathbb{R}$ into $X_{0}^{+}\otimes M_{2}(\mathbb{C})$ and
\begin{equation}\label{eq.Gamma+.0}
    \lim_{x\rightarrow+\infty}
        \left\Vert \Gamma_{+}(x,~\cdot~)\right\Vert
            _{X_{0}^{+}\otimes M_{2}(\mathbb{C})}=0.
\end{equation}
Moreover, the map
\[
    X_{c}^{+} \ni u_{+} \mapsto\left\{  x\mapsto\Gamma_{+}(x,\,\cdot\,)\right\}
        \in C\bigl((c,\infty);X_{0}^{+}\otimes M_{2}(\mathbb{C})\bigr)
\]
is continuous for any $c\in\mathbb{R}$.
\end{lemma}

\begin{proof}
These statements follow from the proof of Proposition 3.5 in \S 3.2 of
Paper~I.
\end{proof}

Write%
\[
    \Psi_{+}(x,k) = \left(\begin{array} [c]{cc}%
        \psi_{11}^{+}(x,k) & \psi_{12}^{+}(x,k)\\
        \psi_{21}^{+}(x,k) & \psi_{22}^{+}(x,k)
                \end{array} \right)
\]
and let
\[
    f_{+}(x,k)=\psi_{11}^{+}(x,k)+\psi_{21}^{+}(x,k).
\]
Then the function~$f_+(\,\cdot\,,k)$ solves the Schr\"{o}dinger equation $-y''+q y=k^2y$ (recall that~$q = u_+' + u_+^2$ by assumption), and from Lemma~\ref{lemma.Gamma+} and the constancy of the Wronskian~(\ref{eq.W}) for solutions, we immediately obtain:

\begin{lemma}
\label{lemma.f+}The representation formulas%
\begin{align}
f_{+}(x,k)  &  =e^{ikx} \left(  1+\int_{0}^{\infty}
    \bigl[\Gamma_{11}^{+}(x,\zeta)+\Gamma_{21}^{+}(x,\zeta)\bigr]
        e^{2ik\zeta}d\zeta\right),\label{eq.f+.rep}\\
f_{+}^{[1]}(x,k)  &  =ik e^{ikx} \left( 1+\int_{0}^{\infty}
    \bigl[\Gamma_{11}^{+}(x,\zeta)-\Gamma_{21}^{+}(x,\zeta)\bigr]
        e^{2ik\zeta}\,d\zeta\right)  \label{eq.f+.quasi.rep}%
\end{align}
hold. Moreover, denoting by $W_{+}$ the modified Wronskian (\ref{eq.W}) with
$u=u_{+}$, we get
\begin{equation}
\label{eq.W+}W_{+}\left\{  f_{+}(x,k),f_{+}(x,-k)\right\}  =-2ik.
\end{equation}
\end{lemma}

This shows that the function~$f_+(\,\cdot\,,k)$ constructed above is indeed the ``right'' Jost solution for the potential~$q\in\mathcal{Q}$. For every fixed~$x\in\mathbb{R}$, we use equalities~\eqref{eq.f+.rep}--\eqref{eq.f+.quasi.rep} to analytically continue $f_+(x,k)$ and $f^{[1]}(x,k)$ for~$k$ in the open complex upper-half plane~$\mathbb{C}^+$. The function~$f_+(\,\cdot\,,k)$ so continued is the Jost solution for the potential~$q$ for all~$k$ in the closed complex upper-half plane~$\overline{\mathbb{C}^{+}}$ and satisfies for such~$k$ the asymptotic condition%
\begin{equation}
    \lim_{x\rightarrow+\infty}
        \left\vert \left(\begin{array}[c]{c}%
            y(x)\\ y^{[1]}(x) \end{array} \right)
        -\left(\begin{array} [c]{c}%
            e^{ikx}\\ ike^{ikx} \end{array}\right)  \right\vert
        =0. \label{eq.Jost+}%
\end{equation}

Similarly, let $\Psi_{-}(x,k)$ be the matrix solution of (\ref{eq.AKNS})\ with
$u=u_{-}$, the Riccati representative of $q$ with $u_{-}\in X^{-}$, that obeys
the asymptotic condition
\[
    \lim_{x\rightarrow-\infty}
        \left\vert \Psi_{-}(x,k)-\exp(ikx\sigma_{3})\right\vert =0.
\]
We then have the representation%
\begin{equation}\label{eq.Psi-.rep}
    \Psi_{-}(x,k)
        = \left(  I+\int_{-\infty}^{0}\Gamma_{-}(x,\zeta)
                \exp(2ik\zeta\sigma_{3})\,d\zeta\right)  \exp(ikx\sigma_{3})
\end{equation}
(taking $z=2k$ in the formulas of Paper~I). In analogy to
Lemma~\ref{lemma.Gamma+}, we have:

\begin{lemma}
\label{lemma.Gamma-}The map $x\mapsto\Gamma_{-}(x,~\cdot~)$ is continuous from
$\mathbb{R}$ into $X_{0}^{-}\otimes M_{2}(\mathbb{C})$ and
\begin{equation}
    \lim_{x\rightarrow-\infty}
        \left\Vert \Gamma_{-}(x,\,\cdot\,)\right\Vert
            _{X_{0}^{-}\otimes M_{2}(\mathbb{C})} = 0. \label{eq.Gamma-.0}%
\end{equation}
Moreover, the map
\[
    X_{c}^{-} \ni u_{-} \mapsto\left\{  x\mapsto\Gamma_{-}(x,\,\cdot\,)\right\}
            \in C\bigl((-\infty,c);X_{0}^{-}\otimes M_{2}(\mathbb{C})\bigr)
\]
is continuous for any $c\in\mathbb{R}$.
\end{lemma}

Writing
\[
    \Psi_{-}(x,k)
     =\left(\begin{array}[c]{cc}%
    \psi_{11}^{-}(x,k) & \psi_{12}^{-}(x,k)\\
    \psi_{21}^{-}(x,k) & \psi_{22}^{-}(x,k)
        \end{array}\right)
\]
and taking
\[
    f_{-}(x,k)=\psi_{12}^{-}(x,k)+\psi_{22}^{-}(x,k)
\]
we find that $f_-(\,\cdot\,,k)$ is a solution of the equation~$-y''+q y=k^2y$ (recall that~$q=u_-'+u_-^2$ by assumption) and obtain the obvious analogue of Lemma \ref{lemma.f+}.

\begin{lemma}
\label{lemma.f-}The representation formulas%
\begin{align}
f_{-}(x,k)  &  =e^{-ikx}\left(  1+\int_{-\infty}^{0}
    \bigl[ \Gamma_{12}^{-}(x,\zeta)+\Gamma_{22}^{-}(x,\zeta)\bigr]  e^{-2ik\zeta}d\zeta\right) ,\label{eq.f-.rep}\\
f_{-}^{\left[  1\right]  }(x,k)  &  =-ike^{-ikx}\left( 1+\int_{-\infty}^{0}
    \bigl[ \Gamma_{12}^{-}(x,\zeta)-\Gamma_{22}^{-}(x,\zeta) \bigr]
    e^{-2ik\zeta}d\zeta\right)  \label{eq.f-.quasi.rep}%
\end{align}
hold. Moreover, if $W_{-}$ denotes the modified Wronskian (\ref{eq.W}) with
$u=u_{-}$, then
\[
    W_{-}\left\{  f_{-}(x,k),f_{-}(x,-k)\right\}  =2ik.
\]
\end{lemma}

Clearly, $f_-(\,\cdot\,,k)$ is the ``left'' Jost solution corresponding to the potential~$q\in\mathcal{Q}$. The analytic extension of~$f_-(x,k)$ for $k\in\mathbb{C}^+$ by means of~\eqref{eq.f-.rep} gives the Jost solution for the potential~$q$ in the closed complex upper-half plane and satisfies for~$k\in \overline{\mathbb{C}^{+}}$ the asymptotic condition
\begin{equation}
    \lim_{x\rightarrow-\infty}
        \left\vert \left(\begin{array} [c]{c}%
            y(x)\\ y^{[1]}(x) \end{array} \right)
        -\left(\begin{array}[c]{c}%
            e^{-ikx}\\ -ike^{-ikx} \end{array}\right)  \right\vert =0. \label{eq.Jost-}%
\end{equation}

It will be important in the analysis of the inverse problem to have the
following equations of Gelfand--Levitan--Marchenko type for $\Gamma_{\pm}$ derived in Proposition 3.8 of Paper~I. Set%
\begin{align}
    F_{+}(x) &  =\frac{1}{\pi}
        \int_{-\infty}^{\infty}e^{2ikx}r_{+}(k)\,dk,\label{eq.F+}\\
    F_{-}(x) &  =\frac{1}{\pi}
        \int_{-\infty}^{\infty}e^{-2ikx}r_{-}(k)\,dk\label{eq.F-}%
\end{align}
(note the slightly different convention for the Fourier transform than in
Paper I owing to the change of variables $s=2k$, and note that $F_{\pm}$ are
real-valued in our case) and define the matrix-valued functions
\[
    \Omega_{-}(x):= \begin{pmatrix} 0 & F_{-}(x)\\ F_{-}(x) & 0 \end{pmatrix},
        \qquad
    \Omega_{+}(x):= \begin{pmatrix} 0 & F_{+}(x)\\ F_{+}(x) & 0 \end{pmatrix}.
\]
Then Proposition~3.8 in Paper~I gives the following relations between
$\Gamma_{\pm}$ and $\Omega_{\pm}$:
\begin{align}
\Gamma_{-}(x,\zeta) + \Omega_{-}(x+\zeta)
    + \int_{-\infty}^{0}\Gamma_{-}(x,t)\,\Omega_{-}(x+t+\zeta)\,dt
    &  =0,\quad\zeta<0,\label{eq.GLM-}\\
\Gamma_{+}(x,\zeta) + \Omega_{+}(x+\zeta)
    + \int_{0}^{\infty}\Gamma_{+}(x,t)\,\Omega_{+}(x+t+\zeta)\,dt
    &  =0,\quad\zeta>0.\label{eq.GLM+}%
\end{align}

Finally, if we write%
\begin{align}
    f_{+}(x,k)  &  =m_{+}(x,k)e^{ikx},\label{eq.fm+}\\
    f_{-}(x,k)  &  =m_{-}(x,k)e^{-ikx}, \label{eq.fm-}%
\end{align}
then (\ref{eq.f+.rep}) and (\ref{eq.f-.rep}) imply that%
\begin{align}
    m_{+}(x,k)  &  =1+\int_{0}^{\infty}K_{+}(x,\zeta)e^{2ik\zeta}\,d\zeta,
        \label{eq.m+}\\
    m_{-}(x,k)  &  =1+\int_{-\infty}^{0}K_{-}(x,\zeta)e^{-2ik\zeta}\,d\zeta,
        \label{eq.m-}%
\end{align}
where%
\begin{align*}
    K_{+}(x,\zeta)  &  =\Gamma_{11}^{+}(x,\zeta)+\Gamma_{21}^{+}(x,\zeta),\\
    K_{-}(x,\zeta)  &  =\Gamma_{12}^{-}(x,\zeta)+\Gamma_{22}^{-}(x,\zeta).
\end{align*}
Therefore, by passing to Fourier transforms of (\ref{eq.GLM-})--(\ref{eq.GLM+}) in the usual way, we immediately obtain:

\begin{lemma}\label{lemma.hardy}
For each fixed $x\in\mathbb{R}$, the functions
\[
    m_{-}(x,-k)+e^{-2ikx}r_{-}(k)m_{-}(x,k)
\]
and
\[
    m_{+}(x,-k)+e^{ 2ikx}r_{+}(k)m_{+}(x,k)
\]
of variable~$k$ belong to~$H_{+}^{2}(\mathbb{R})$.
\end{lemma}

We conclude this subsection with the observation that
if $q$ is a Miura potential with extremal
Riccati representatives~$u_{\pm}\in X^{\pm}$, then the Jost solutions at zero
energy are given by%
\begin{align*}
    f_{+}(x,0)  &  =\exp\left( - \int_{x}^{\infty}u_{+}(y)\,dy\right)  ,\\
    f_{-}(x,0)  &  =\exp\left(   \int_{-\infty}^{x}u_{-}(y)\,dy\right)
\end{align*}
and, conversely,%
\begin{align}
    u_{+}(x)  &  =\frac{f_{+}^{\prime}(x,0)}{f_{+}(x,0)},\label{eq.u+.jost}\\
    u_{-}(x)  &  =\frac{f_{-}^{\prime}(x,0)}{f_{-}(x,0)}. \label{eq.u-.jost}%
\end{align}

\subsection{Representation formulas for the scattering coefficients}
\label{subsec.rep}

The transmission and reflection coefficients can be computed from Wronskians
of the Jost solution, with due care given to the fact that $W_{+}$ and $W_{-}$
are \emph{different} modified Wronskians corresponding to \emph{distinct}
Riccati representatives of $q$ if $u_{+}\neq u_{-}$.

Recall that the coefficients $a(k)$ and $b(k)$ were defined by the usual
relations~(\ref{eq.ab+}) or \eqref{eq.ab-}. Using~\eqref{eq.W+}, one derives the formulas
\begin{align}
    a(k) &  =\frac{W_{+}\left\{  f_{-}(x,k),f_{+}(x,k)\right\}  }{2ik},
        \label{eq.a.wronski}\\
    b(k) &  =-\frac{W_{+}\left\{  f_{-}(x,k),f_{+}(x,-k)\right\}}{2ik}
        \label{eq.b.wronski}%
\end{align}
for real nonzero~$k$. Since $\overline{f_{\pm}(\cdot,k)}=f_{\pm}(\cdot,-k)$ for
such~$k$ and since $u$ is real-valued, it follows that $a$ and $b$ have the
property
\begin{align}
    \overline{a(k)} &  =a(-k),\label{eq.abar}\\
    \overline{b(k)} &  =b(-k),\label{eq.bbar}%
\end{align}
while, by standard arguments,%
\begin{equation}\label{eq.ab1}%
    \left\vert a(k)\right\vert ^{2}-\left\vert b(k)\right\vert ^{2}=1.
\end{equation}
We also note that~\eqref{eq.a.wronski} can be used to extend $a$ analytically
to the open upper-half plane~$\mathbb{C}^{+}$. The function~$a$ so defined is
continuous in~$\overline{\mathbb{C}^{+}}\setminus\{0\}$ and has no zeros
there. Indeed, $a(k)\neq0$ for $k\in\mathbb{R}\setminus\{0\}$ due
to~\eqref{eq.ab1}. Assume that there is $k\in\mathbb{C}^{+}$ such that
$a(k)=0$. Then the Jost solutions $f_{+}(\cdot,k)$ and $f_{-}(\cdot,k)$ are
linearly dependent and thus equation~\eqref{eq.k2} has a solution that decays
exponentially at $\pm\infty$. Therefore $k^{2}\in\mathbb{C}\setminus
\mathbb{R}^{+}$ is an eigenvalue of the Schr\"{o}dinger operator~$S$, which is
impossible in view of non-negativity of the latter.

Using the Wronskian formulas (\ref{eq.a.wronski})--(\ref{eq.b.wronski}) at
$x=0$, together with the representation formulas (\ref{eq.f+.rep})--(\ref{eq.f+.quasi.rep}) and (\ref{eq.f-.rep})--(\ref{eq.f-.quasi.rep}), we
derive the following representation for~$a$ and~$b$ (recall that $\widehat f$
stands for the Fourier transform of a function~$f$ normalized by~\eqref{eq.Fourier}).

\begin{lemma}
\label{lemma.ab} Suppose that $q\in\mathcal{Q}$. Then the coefficients $a$ and $b$ admit the representation
\begin{align}
    a(k)&=1-\widehat{A}_{1}(k)+v(0)\left[  \frac{1-\widehat{A}_{2}(k)}{2ik}\right],
\label{eq.ak} \\
    b(k)&=\widehat{B}_{1}(k)-v(0)\left[  \frac{1-\widehat{B}_{2}(k)}{2ik}\right],
\label{eq.bk}%
\end{align}
in which $A_{j}$ and $B_{j}$, $j=1,2,$, are real-valued functions in~$X$, with~$A_{i}$ supported on~$[0,\infty)$. Moreover, $A_2 = B_2 = 0$ if $q\in\mathcal{Q}_0$ and
\[
    1-\widehat{A}_{2}(0)=1-\widehat{B}_{2}(0)=f_{+}(0,0)f_{-}(0,0)
\]
is nonzero if $q\in\mathcal{Q}_{>0}$. The maps $q \mapsto A_{j}$ and $q \mapsto B_{j}$ are continuous maps from $\mathcal{Q}$ into~$X$.
\end{lemma}

The derivation of (\ref{eq.ak}) and (\ref{eq.bk}) uses the fact that $\widehat X$ is a Banach algebra and is quite straightforward. We omit the corresponding calculations and only note that the term involving $v(0)$ arises from the fact that the
quasi-derivatives in $W_{+}$ are referred to $u_{+}$, whereas the
representation formula~\eqref{eq.f-.quasi.rep} for $f_{-}^{\left[  1\right]  }$ refers to $u_{-}$. Since $u_{+}$ and $u_{-}$ differ by a continuous function $v$ and $f_{-}$ is absolutely continuous in $x$, the expression $f_{-}^{\prime}(x,k) -
u_{+}(x)f_{-}(x,k)=f_{-}^{\left[  1\right]  }(x,k)-v(x)f_{-}(x,k)$ defines a
continuous function.

\begin{corollary}\label{cor.pole-of-a}
If $v(0)\neq0$, i.e., if $q\in\mathcal{Q}_{>0}$, then
the function~$a$ has a singularity at $k=0$ with%
\begin{equation}\label{eq.theta}%
    \theta:=\lim_{k\rightarrow0}2ik\,a(k)=v(0)f_{+}(0,0)f_{-}(0,0)>0.
\end{equation}
\end{corollary}

\begin{remark}
If $q$ is sufficiently regular so that the ordinary derivatives $f_{+}^{\prime}(x,k)$ and $f_{-}^{\prime}(x,k)$ in~$x$ exist and are continuous, then relations~\eqref{eq.u+.jost}--\eqref{eq.u-.jost} yield the representation of the number~$\theta$ of~\eqref{eq.theta} as the usual Wronskian of the Jost solutions, viz.
\[
\theta=f_{+}^{\prime}(0,0)f_{-}(0,0)-f_{+}(0,0)f_{-}^{\prime}(0,0).
\]
\end{remark}

We now turn to the reflection and transmission coefficients defined by
(\ref{eq.r+})--(\ref{eq.t}).

\begin{proposition}\label{prop.maps-to-R-neq-0}
Suppose that $q\in\mathcal{Q}_{>0}$. Then $r_{\pm}\in\mathcal{R}_{>0}$, and the maps (\ref{eq.Spm}) and (\ref{eq.tmap}) are continuous.
\end{proposition}

\begin{proof}
First, $r_{\pm}(-k)=\overline{r_{\pm}(k)}$ by (\ref{eq.abar})--(\ref{eq.bbar}), and $\left\vert r_{\pm}(k)\right\vert <1$ for $k\neq0$ by~(\ref{eq.ab1}).
To show that $r_{\pm}\in\widehat{X}$ and to analyze small-$k$ behavior we
factor out the leading behavior of $a(k)$ and $b(k)$ as $k\rightarrow0$. The
function $\left(  k+i\right)  ^{-1}$ is easily seen to belong to $\widehat{X}%
$, so that, by Lemma~\ref{lemma.ab}, Corollary~\ref{cor.pole-of-a}, and the
Banach algebra structure of~$\widehat{X}$, the functions%
\begin{equation}\label{eq.abtilde}
\begin{aligned}
    \widetilde{a}(k)  &  =\frac{k}{k+i}a(k),\\
    \widetilde{b}(k)  &  =\frac{k}{k+i}b(k)
\end{aligned}
\end{equation}
also belong to $1\dotplus\widehat{X}$ and $\widehat{X}$, respectively, and the
map $q  \mapsto\bigl(\,\widetilde{a},\widetilde{b}\,\bigr)$ is continuous. Moreover, the function~$\widetilde a$ does not vanish on~$\mathbb{R}$, tends to~$1$ at infinity and thus is an invertible element of $1\dotplus\widehat{X}$. Observing that (cf.~(\ref{eq.r+})--(\ref{eq.t}))
\begin{align}
    r_{+}(k)  &  =\frac{i-k}{i+k}\frac{\widetilde{b}(-k)}{\widetilde{a}(k)},
\label{eq.r+.renorm}\\
    r_{-}(k)  &  =\frac{\widetilde{b}(k)}{\widetilde{a}(k)}, \label{eq.r-.renorm}%
\end{align}
we see that $r_{\pm}\in\widehat{X}$. \ Since $\widetilde{a}$ and $\widetilde{b}$ are continuous and
\[
    \widetilde{a}(0)=-\widetilde{b}(0)=-\tfrac12v(0)f_{+}(0,0)f_{-}(0,0)\ne0,
\]
we have $r_{\pm}(0)=-1$. Finally, we compute%
\begin{equation}
    \widetilde{r}_\pm (k) =
        \frac{1-\left\vert r_{\pm}(k)\right\vert ^{2}}{k^{2}}
            =\frac{1}{k^{2}+1}
            \frac{1}{\left\vert \widetilde{a}(k)\right\vert ^{2}}; \label{eq.thing.renorm}%
\end{equation}
since $\left\vert \widetilde{a}(k)\right\vert ^{2} = \widetilde{a}(k)\widetilde{a}(-k)$ belongs to~$1\dotplus\widehat{X}$ and is an invertible element there, the left-hand side of~\eqref{eq.thing.renorm} belongs to
$\widehat{X}$ and satisfies the condition~$\widetilde{r}_\pm(0)>0$.

Continuity of the maps~\eqref{eq.Spm} and \eqref{eq.tmap} as
maps from $X_{0}^{+}\times X_{0}^{-}\times\mathbb{R}^+$ into $\widehat{X}$
follows from Lemma~\ref{lemma.ab} and formulas \eqref{eq.r+.renorm}--\eqref{eq.thing.renorm}.
\end{proof}

To complete the proof of Theorem \ref{thm.direct}, we need to show that a
potential $q\in\mathcal{Q}$ is uniquely determined by a single reflection
coefficient. The following proof is a simple variant of the proof of
Levinson's theorem given in \cite{DT:1979}.

\begin{proposition}
\label{prop.one-to-one}Suppose that $q_{1}, q_{2}\in\mathcal{Q}$ have the same
right reflection coefficient $r$. Then $q_{1}=q_{2}$ as distributions.
\end{proposition}

\begin{proof}
First, we recall the vanishing lemma (Lemma 1 on p. 207 of \cite{DT:1979})
which states that if $h\in H_{+}^{2}(\mathbb{R})$, if $r\in L^{\infty}(\mathbb{R})$ with $\left\vert r(k)\right\vert <1$ a.e., and
 $rh+\overline{h}\in H_{+}^{2}(\mathbb{R})$,
then $h=0$. As in \cite{DT:1979} we will use the analyticity property of~$m_{\pm}$ given in Lemma \ref{lemma.hardy} to prove that $m_{+}(x,k,q_{1})=m_{+}(x,k,q_{2})$ and $m_{-}(x,k,q_{1}) = m_{-}(x,k,q_{2})$. If so, we can conclude from (\ref{eq.u+.jost})--(\ref{eq.u-.jost}) that the left and right Riccati representatives of the two potentials are the same; hence $v(0;q_{1})=v(0;q_{2})$ and $q_{1}=q_{2}$ as distributions.

To prove the equality of Jost solutions, fix $x\in\mathbb{R}$ and let
\begin{align*}
    h_{+}(k)  &  =m_{+}(x,k,q_{1})-m_{+}(x,k,q_{2}),\\
    h_{-}(k)  &  =m_{-}(x,k,q_{1})-m_{-}(x,k,q_{2}).
\end{align*}
Then $h_{\pm}\in H_{+}^{2}(\mathbb{R})$ and, by Lemma \ref{lemma.hardy},
\begin{align*}
    h_{+}(-k)+e^{2ikx}r_{+}(k)h_{+}(k)  &  \in H_{+}^{2}(\mathbb{R}),\\
    h_{-}(-k)+e^{-2ikx}r_{-}(k)h_{-}(k)  &  \in H_{+}^{2}(\mathbb{R}).
\end{align*}
Since $h_{\pm}(-k)=\overline{h_\pm(k)}$, it follows from the vanishing lemma that $h_{+}=h_{-}=0$.
\end{proof}

Propositions \ref{prop.maps-to-R-neq-0} and \ref{prop.one-to-one} give the
proof of Theorem \ref{thm.direct}.


\subsection{Reconstruction of the transmission coefficient}

\label{subsec.recon}

Next, we show how to construct $t(k)$ given either one of the reflection
coefficients. As we mentioned above, formula~\eqref{eq.a.wronski} allows to extend~$a$ analytically in the open upper-half plane~$\mathbb{C}^+$ and this extension has no zeros in~$\overline{\mathbb{C}^{+}}\setminus\{0\}$. Thus the regularization~$\widetilde a$ of~$a$ given by~\eqref{eq.abtilde} extends to
a bounded holomorphic function in the upper-half plane with no zeros in its
closure. Using the Schwarz formula to reconstruct the function $\log
\widetilde{a}$ from its real part $\operatorname{Re}\log\widetilde{a}%
(s)=\log|\widetilde{a}(s)|$, we get
\begin{equation}
    \widetilde{a}(z)=\exp\left(  \frac{1}{\pi i}\int_{\mathbb{R}}\log\left\vert
    \widetilde{a}(s)\right\vert \frac{ds}{s-z}\right)  .\label{eq.tildea.rep}%
\end{equation}
Combining \eqref{eq.thing.renorm}--(\ref{eq.tildea.rep}) results in%
\[
t(z)  :=1/a(z)
        = \frac{z}{z+i}
        \exp\left\{  \frac{1}{2\pi i}\int_{\mathbb{R}}
            \log\left[\left(  1-\left\vert r_{\pm}(s)\right\vert^{2} \right) \frac{s^2+1}{s^2} \right]
            \frac{ds}{s-z}\right\},
\]
and $t$ on the real line is given as a boundary value as $\operatorname{Im} z\rightarrow 0$. Recalling the Riesz projector~$\mathcal{C}_+$ of~\eqref{eq.Cauchy}, we get the formula
\begin{equation}\label{eq.t.recon}
    t(k) = \frac{k}{k+i}
        \exp \left\{  \left(\mathcal{C}_+
       \log \Bigl[ \left(1-\left\vert r_\pm(s)\right\vert ^{2}\right)\frac{s^2+1}{s^2} \Bigr]\right)(k)\right\}.
\end{equation}
In particular, the number~$\theta:=\lim_{k\rightarrow0}\left[2ik/t(k)\right]$ can be recovered from either reflection coefficient.

We next show that formula~\eqref{eq.t.recon} makes sense for every element $r\in\mathcal{R}_{>0}$. Indeed, as noted in Introduction, the function
\begin{equation}\label{eq.r.tilde}
    \left(1-\left\vert r(k)\right\vert ^{2}\right)\frac{k^2+1}{k^2}
        = 1-r(k)r(-k) + \widetilde{r}(k)
\end{equation}
belongs to the algebra~$1 \dotplus \widehat{X}$, does not vanish on the real line, and tends to~$1$ at infinity. By the Wiener--Levi lemma (Lemma~A.2 of Paper~I), the function
\[
    \log \Bigl[\left(1-\left|r(k)\right|^{2}\right)\frac{k^2+1}{k^2}\Bigr]
\]
also belongs to $1\dotplus \widehat X$; in fact, since it vanishes at infinity, it belongs to $\widehat X$. Finally, the Riesz projector~$\mathcal{C}_+$ acts continuously in~$\widehat X$, and exponentiation is a continuous operation in $1 \dotplus \widehat X$ by the Wiener--Levi lemma. We now define a function~$\widetilde t\in 1 \dotplus \widehat X$ by (cf.~\eqref{eq.t.recon})
\begin{equation}\label{eq.tildet.rep}
    \widetilde t = \exp \left\{ \mathcal{C}_+
        \log \Bigl[\left(1-\left|r(k)\right|^{2}\right)\frac{k^2+1}{k^2}\Bigr]
        \right\}.
\end{equation}
Clearly, $\widetilde t$ is an invertible element of the Banach algebra~$1 \dotplus \widehat X$. Moreover, the following holds:

\begin{lemma}\label{lemma.t.recon}
The mappings
\[
     \mathcal{R}_{>0} \ni r \mapsto \widetilde t \in 1 \dotplus \widehat X
\]
and
\[
     \mathcal{R}_{>0} \ni r \mapsto 1/\widetilde t \in 1 \dotplus \widehat X
\]
are continuous.
\end{lemma}

\begin{proof}
Since taking the inverse in a Banach algebra~$\mathcal{A}$ is a continuous operation on the open set of all invertible elements of~$\mathcal{A}$, only the first mapping needs to be studied. As noted above, the Riesz projector acts continuously in $\widehat X$ and exponentiation is a continuous mapping from $\widehat X$ to $1 \dotplus \widehat X$, so that it suffices to prove that the mapping
\[
    \mathcal{R}_{>0} \ni r \mapsto
        \log \Bigl[\left(1-\left|r(k)\right|^{2}\right)\frac{k^2+1}{k^2}\Bigr]
            \in \widehat X
\]
is continuous.

Fix an arbitrary~$r_0\in \mathcal{R}_{>0}$ and set
\[
    \varepsilon_0:= \inf_{|k|>1} (1-|r_0(k)|^2),
    \qquad
    \varepsilon_1:= \inf_{|k|\le1} \widetilde r_0(k).
\]
Since $r_0$ is a continuous functions vanishing at infinity and satisfying $|r_0(k)|<1$ for $|k|>1$, we get $\varepsilon_0>0$; also $\varepsilon_1>0$ since $\widetilde r_0$ is a positive continuous function. Fix now $\varepsilon>0$ that is less than $\min\{\varepsilon_0,\varepsilon_1\}/4$ and take an $\varepsilon$-neighbourhood $\mathcal{O}(r_0)$ of the point~$r_0$ in~$\mathcal{R}_{>0}$. Observing that the norm in~$\widehat X$ dominates the sup-norm, i.e., that
\[
    \sup_{k\in\mathbb{R}}|f(k)| \le \|f\|_{\widehat X}
\]
for every~$f\in\widehat X$, we conclude that
\[
    \inf_{|k|>1} (1-|r(k)|^2) \ge \varepsilon_0/2
\]
and
\[
    \inf_{|k|\le1} \widetilde r(k) \ge \varepsilon_1/2
\]
for every $r\in \mathcal{O}(r_0)$. It follows that
\[
    \left(1-\left|r(k)\right|^{2}\right)\frac{k^2+1}{k^2}
        \ge 2\varepsilon
\]
for all $k\in\mathbb R$ and all $r \in \mathcal{O}(r_0)$. The Wiener--Levi lemma (Lemma~A.2 of Paper~I) now yields continuity (in fact, even analyticity) of the mapping
\[
    \mathcal{O}(r_0) \ni r \mapsto \log \Bigl[\left(1-\left|r(k)\right|^{2}\right)\frac{k^2+1}{k^2}\Bigr]
            \in \widehat X.
\]
The proof is complete.
\end{proof}

We observe that since the function in~\eqref{eq.r.tilde} is even and the Riesz projector maps even functions into odd ones, the function~$\widetilde t$ enjoys the symmetry property~$\widetilde t(-k) = \overline{\widetilde t(k)}$. We set $t(k)= k \widetilde t(k)/(k+i)$; then the above considerations show that
\[
    \frac{t(k)}{t(-k)}
        = \frac{k-i}{k+i}\frac{\widetilde{t}(k)}{\widetilde{t}(-k)}%
\]
also belongs to $1\dotplus\widehat{X}$. The function%
\[
    r^{\#}(k)=-\frac{t(k)}{t(-k)}r(k)
\]
thus belongs to $\mathcal{R}$ and, as $\left\vert t(k)/t(-k)\right\vert =1$,
we have%
\[
    \frac{1-\left\vert r^{\#}(k)\right\vert ^{2}}{k^{2}}
        = \frac{1-\left\vert r(k)\right\vert ^{2}}{k^{2}}\in\widehat{X}.
\]
Hence:

\begin{proposition}\label{pro.I}
For $r\in\mathcal{R}_{>0}$, define $\widetilde t$ by~(\ref{eq.tildet.rep}) and set $t(k)= k \,\widetilde t(k)/(k+i)$. Then the nonlinear map%
\[
    \mathcal{I}\,:\, r\mapsto r^{\#}(k):= -\frac{t(k)}{t(-k)}r(-k)
\]
is a continuous involution on $\mathcal{R}_{>0}$.
\end{proposition}

\section{The inverse problem}

\label{sec.inverse}

In this section, we prove:

\begin{theorem}\label{thm.inverse}
Suppose that $r\in\mathcal{R}_{>0}$. There exists a unique
$q\in\mathcal{Q}_{>0}$ having $r$ as its right reflection coefficient.
Moreover, the map $r\mapsto q$ is continuous.
\end{theorem}

We suppose given a function $r\in\mathcal{R}_{>0}$, presumed to be the right
reflection coefficient corresponding to a potential~$q_0$ to be found. From this
data, we can construct~$t(k)$ (and hence $a(k):=1/t(k)$) using (\ref{eq.t.recon}), and use the involution $\mathcal{I}$ to construct $r^{\#}=\mathcal{I}r$, a
candidate for the left reflection coefficient. Clearly, we then should define $b$ as $r^{\#} a$.

In \S \ref{subsec.GLM}, we form two Gelfand--Levitan--Marchenko equations like~\eqref{eq.GLM-}--\eqref{eq.GLM+} of Subsection~\ref{subsec.jost} but taking the putative reflection coefficients~$r$ and $r^{\#}$ instead of~$r_+$ and~$r_-$ and prove that these equations are uniquely soluble for the kernels~$\Gamma$ and~$\Gamma^{\#}$. These kernels determine candidate right and left Riccati representatives~$w$ and~$w^{\#}$, which give the Riccati data
\begin{equation}
\bigl(  \left.  w\right\vert _{\mathbb{R}^+},
         \left.  w^{\#}\right\vert_{\mathbb{R^-}},(w^{\#}-w)(0)\bigr) \label{eq.riccati.data}%
\end{equation}
of a distribution potential $q_0 \in\mathcal{Q}$. The
construction exhibits continuity of the map from $r$ to the data
(\ref{eq.riccati.data}) as maps from $X$ to $X_{0}^{+}\times X_{0}^{-}%
\times\mathbb{R}^+$.

In \S \ref{subsec.recon.jost} we show that $\Gamma$ and $\Gamma^{\#}$ can be used to construct Jost solutions for half-line Schr\"{o}dinger operators with Riccati representatives $w$ and $w^{\#}$ respectively. Finally, in \S \ref{subsec.consistent} we justify the reconstruction by showing that $w'+w^2=(w^{\#})'+(w^{\#})^2=q_0$ and that $q_0$ has reflection coefficients $r$ and $r^{\#}$. It then follows from Proposition~\ref{prop.one-to-one} that $q_0$ is the correct reconstruction.

\subsection{Reconstruction maps}\label{subsec.GLM}

Given a function $r\in\mathcal{R}_{>0}$, we compute $t=1/a$ as a boundary value of (\ref{eq.t.recon}) and set $r^{\#}:=\mathcal{I}r$ and $b:= r^{\#}a$.

Next, set%
\begin{align}\label{eq.F}
    F(x)  &  :=\frac{1}{\pi}\int_{-\infty}^{\infty}r(k)e^{2ikx}\,dk,\\
    F^{\#}(x)  &  :=\frac{1}{\pi}\int_{-\infty}^{\infty}r^{\#}(k)e^{-2ikx}\,dk,
\label{eq.Fsharp}%
\end{align}
form the matrix-valued functions
\[
    \Omega(x) := \begin{pmatrix} 0 & F(x)\\ F(x) & 0 \end{pmatrix},  \qquad\Omega^{\#}(x)
        := \begin{pmatrix} 0 & F^{\#}(x)\\ F^{\#}(x) & 0 \end{pmatrix},
\]
and consider the integral equations (cf.\ equations~\eqref{eq.GLM-}--\eqref{eq.GLM+})
\begin{align}
    \Gamma^{\#}(x,\zeta) + \Omega^{\#}(x+\zeta)
        + \int_{-\infty}^{0}\Gamma^{\#}(x,t)\Omega^{\#}(x+t+\zeta)\,dt  & =0, \quad\zeta<0,\label{eq.GLM-sharp}%
\\
    \Gamma(x,\zeta) + \Omega(x+\zeta)
        + \int_{0}^{\infty} \Gamma(x,t)\Omega(x+t+\zeta)\,dt  & =0, \quad\zeta>0. \label{eq.GLM}%
\end{align}
The matrix equation~\eqref{eq.GLM} yields the following system for the
entries~$\Gamma_{11}$ and $\Gamma_{12}$ of~$\Gamma$:
\begin{alignat}{2}
\label{eq.Gamma.11}
                 & \Gamma_{11}(x,\zeta)
            + \int_{0}^{\infty}\Gamma_{12}(x,t)F(x+t+\zeta)\,dt  &= 0,\\
    F(x+\zeta) \, +\,& \Gamma_{12}(x,\zeta)
            + \int_{0}^{\infty}\Gamma_{11}(x,t)F(x+t+\zeta)\,dt  & =0, \label{eq.Gamma.12}%
\end{alignat}
which gives a single equation to determine~$\Gamma_{12}$ in the form
\begin{equation}\label{eq.Gamma}
\begin{aligned}
    \Gamma_{12}(x,\zeta) & +F(x+\zeta) \\
        & -\int_{0}^{\infty}\int_{0}^{\infty}
            \Gamma_{12}(x,t_{1}) F(x+t_{1}+t_{2})F(x+t_{2}+\zeta)\,dt_{1}\,dt_{2}
        =0.
\end{aligned}
\end{equation}
Similar reasoning results in the following equation for $\Gamma^{\#}_{12}$:
\begin{equation}\label{eq.GammaSharp}
\begin{aligned}
    \Gamma^{\#}_{12}(x,\zeta) & +F^{\#}(x+\zeta) \\
        & -\int_{-\infty}^{0}\int_{-\infty}^{0}
            \Gamma^{\#}_{12}(x,t_{1}) F^{\#}(x+t_{1}+t_{2}) F^{\#}(x+t_{2}+\zeta)\,dt_{1}\,dt_{2}
        =0.
\end{aligned}
\end{equation}

We now show that, for each $x$, equations~\eqref{eq.Gamma} and
\eqref{eq.GammaSharp} have unique solutions $\Gamma_{12}(x,~\cdot~)$ and
$\Gamma^{\#}_{12}\left(  x,~\cdot~\right)  $ belonging to $X_{0}^{+}$ and
$X_{0}^{-}$ and then take (cf.~Proposition~3.9 of Paper~I)
\begin{align}
    w(x)       &  :=-\Gamma_{12}(x,0),\label{eq.w}\\
    w^{\#}(x)  &  :=\Gamma^{\#}_{12}(x,0) \label{eq.wsharp}%
\end{align}
as putative right and left Riccati representatives of a potential~$q$ to be found. Denote by $(X_{\mathbb{R}})_1$ the set of real-valued functions $F\in X$ such that $|\widehat F|<1$~a.e.

\begin{proposition}\label{pro.Gamma}
For each $x\in\mathbb{R}$, the equations (\ref{eq.Gamma})
and (\ref{eq.GammaSharp}) have unique solutions respectively in $X_{0}^{+}$
and $X_{0}^{-}$. These solutions depend therein continuously on $x\in\mathbb{R}$ and $F,F^{\#}\in (X_{\mathbb{R}})_1$, and can be written as
\begin{align*}
    \Gamma_{12}(x,\zeta)       &  =-F(x+\zeta)-G(x,\zeta)\\
    \Gamma^{\#}_{12}(x,\zeta)  &  =-F^{\#}(x+\zeta)-G^{\#}(x,\zeta)
\end{align*}
where $G$ (resp. $G^{\#}$) is jointly continuous in $x$ and $\zeta$. Moreover,
the functions $w$ and $w^{\#}$ defined by (\ref{eq.w}) and (\ref{eq.wsharp})
have the following properties:

\begin{itemize}
\item[(1)] For any $c\in\mathbb{R}$, $w\in X_{c}^{+}$ and $w^{\#}\in X_{c}^{-}$.
\item[(2)] The maps
\[
    (X_{\mathbb{R}})_1 \ni F \mapsto w \in X_{c}^{+}
\]
and
\[
    (X_{\mathbb{R}})_1 \ni F^{\#} \mapsto w^{\#} \in X_{c}^{-}\newline%
\]
are continuous for each $c\in\mathbb{R}$.
\item[(3)] $w^{\#}-w$ is a continuous function.
\end{itemize}
\end{proposition}

\begin{proof}
The proofs of all statements but (3) are almost identical to the proofs of
Lemmas~4.1--4.3 in Paper~I. We will only discuss the few changes needed in
this case, using the notation of Paper~I, and we will only discuss
(\ref{eq.Gamma}) and (\ref{eq.w}) since the corresponding proofs for
(\ref{eq.GammaSharp}) and (\ref{eq.wsharp}) are very similar.

If we introduce the linear operator~$T_{F}(x)$ on~$X_{0}^{+}$ by%
\[
    T_{F}(x)\psi(y):=\int_{0}^{\infty}\psi(t)F(x+y+t)~dt,
\]
then equation (\ref{eq.Gamma}) takes the form
\[
    (I-T^{2}_{F}(x)) \Gamma_{12}(x,\,\cdot\,) = - F(x+\,\cdot\,),
\]
and thus to solve~\eqref{eq.Gamma} we need to study properties of~$T_{F}(x)$.
This operator obeys the estimates (4.3) in Paper~I except that the inequality
\begin{equation}\label{eq.TFx}
    \left\Vert T_{F}(x)\right\Vert _{L^{2}\rightarrow L^{2}} \le
    \left\Vert r\right\Vert _{\infty}
\end{equation}
in view of the relation $r(0)=-1$ does not show that the $L^2$-norm of $T_{F}(x)$ is less than~$1$. We shall use a different argument proving that~$I- T^2_{F}(x)$ is invertible in~$X^+_0$ and that for every~$c\in\mathbb{R}$ there exists a constant~$\rho_c<1$ such that $\|T_{F}(x)\|_{L^{2}\rightarrow L^{2}}<\rho_c$ for all~$x\in \mathbb{R}^+_c$.

To do this, one first shows that $\ker_{L^{2}}(I\pm T_{F}(x))$ is trivial, exploiting the fact that $\left\vert r(k)\right\vert <1$ a.e.\ (see, for example, the
proof of Lemma~6.4.1 in~\cite{Marchenko:1977}). Since $T_{F}(x)$ maps
boundedly $L^{1}(\mathbb{R}^{+})$ into $L^{2}(\mathbb{R}^{+})$, we also have
that $\ker_{L^{1}}(I\pm T_{F}(x))$ is trivial. The operator~$T_{F}(x)$ is compact in $L^1(\mathbb{R}^+)$ and in~$L^2(\mathbb{R}^+)$; this is obviously true if $F\in C_{0}^{\infty}(\mathbb{R})$, hence follows for an arbitrary $F\in X$ by the norm-closure of the compact operators and an approximation argument (based on the $L^1$-norm estimate~(4.3) of Paper~I and the inequality~\eqref{eq.TFx} together with the fact that~$\|r\|_\infty\le \|F\|_X$). The Fredholm alternative implies that the operator~$I-T_{F}^2(x)$ is boundedly invertible in~$X^+_0$.

Next, since $T_F(x)$ is a self-adjoint operator in~$L^2(\mathbb{R^+})$ and $\left\Vert T_{F}(x)\right\Vert_{L^{2}\rightarrow L^{2}} \le1$ by~\eqref{eq.TFx}, for every $c \in \mathbb{R}$ there exists $\rho_c<1$ such that the spectrum of $T_F(c)$ belongs to~$[-\rho_c,\rho_c]$. The explicit dependence of~$T_{F}(x)$ on~$x$ shows that the {$L^2$-norm} of~$T_{F}(x)$ is a non-increasing function of~$x$ and thus
\[
    \left\Vert T_{F}(x)\right\Vert _{L^{2}\rightarrow L^{2}} \le \rho_c
\]
for all $x \in \mathbb{R}^+_c$ as claimed.

One then follows the proofs of Lemmas 4.1--4.3 of
Paper~I (replacing therein $\|r\|_\infty$ and $\rho$ with $\rho_c$) to obtain the claimed properties of $\Gamma_{12}$ and $w$.

To prove statement (3) we note that%
\begin{align*}
    w(x)      &  =F(x)+G(x,0)\\
    w^{\#}(x) &  =-F^{\#}(x)-G^{\#}(x,0)
\end{align*}
so it suffices to show that $F+F^{\#}$ is continuous. By~(\ref{eq.F})--(\ref{eq.Fsharp}),
\begin{align*}
    F(x) + F^{\#}(x)
        & =\frac{1}{\pi}\int_{-\infty}^{\infty}
            \left(r(k)+r^{\#}(-k)\right)  e^{2ikx}\,dk\\
        & =\frac{1}{\pi}\int_{-\infty}^{\infty}
            \left(  1-\frac{t(k)}{t(-k)}\right) r(k)e^{2ikx}\,dk.
\end{align*}
Since the function~$t(k)/t(-k)$ is in~$1\dotplus\widehat{X}$ and tends to $1$ at infinity, it follows that $\left(1-\frac{t(k)}{t(-k)}\right)  r(k)\in L^{1}(\mathbb{R})$ and $F+F^{\#}$ is continuous.
\end{proof}

Formula~\eqref{eq.Gamma.11} and its analogue for the kernel~$\Gamma^{\#}$ show that the entries $\Gamma_{11}$ and $\Gamma^{\#}_{11}$ of the solutions $\Gamma$ and $\Gamma^{\#}$ have similar continuous dependence on~$F$ and $F^{\#}$. Also, the symmetry of $\Omega$ yields the relations $\Gamma_{21}=\Gamma_{12}$ and $\Gamma_{22}=\Gamma_{11}$; similar arguments give the equalities $\Gamma^{\#}_{21}=\Gamma^{\#}_{12}$ and $\Gamma^{\#}_{22}=\Gamma^{\#}_{11}$.

\subsection{Jost solutions}\label{subsec.recon.jost}

In this section we shall show that the matrix-valued kernels $\Gamma$ and
$\Gamma^{\#}$ may be used as in~\eqref{eq.Psi+.rep} and \eqref{eq.Psi-.rep}
to construct solutions of the ZS-AKNS systems~\eqref{eq.AKNS} with $u=w$ and
$u=w^{\#}$. We start with the following observation.

\begin{lemma}\label{lemma.Gamma.partial}
Assume that $r\in \mathcal{R}_{>0}$ is such that the corresponding function~$F$ belongs to the Sobolev space~$H^2(\mathbb{R})$. Then the solutions~$\Gamma_{11}$ and $\Gamma_{12}$ of
the system~\eqref{eq.Gamma.11}--\eqref{eq.Gamma.12} satisfy the system of equations
\begin{align}\label{eq.Gamma.11.partial}
    \frac{\partial}{\partial x}\Gamma_{11}(x,\zeta)
        & = w(x) \Gamma_{12}(x,\zeta),\\
    \biggl(\frac{\partial}{\partial x} - \frac{\partial}{\partial\zeta}\biggr) \Gamma_{12}(x,\zeta)
        & = w(x) \Gamma_{11}(x,\zeta).
    \label{eq.Gamma.21.partial}%
\end{align}
\end{lemma}

\begin{proof}
Under the assumptions of the lemma, the solutions~$\Gamma_{11}$ and $\Gamma_{12}$ are continuously differentiable in $x$ and $\zeta$. Indeed, smoothness in $\zeta$ is obtained immediately from the equations~\eqref{eq.Gamma.11} and~\eqref{eq.Gamma.12}, while the way $T_{F}(x)$ depends on~$x$ and the arguments given in the proof of Proposition~\ref{pro.Gamma} show continuous differentiability in~$x$.

Differentiating now~\eqref{eq.Gamma.11} in~$x$ and then integrating by parts on
account of~\eqref{eq.w} yields
\[
    \frac{\partial}{\partial x}\Gamma_{11}(x,\zeta)
        = -w(x)F(x+t) - \int_{0}^{\infty}
            \biggl(\frac{\partial\Gamma_{12}}{\partial x}
                - \frac {\partial\Gamma_{12}}{\partial\zeta}\biggr) (x,t)F(x+t+\zeta)\,dt .
\]
On the other hand, we find from equation~\eqref{eq.Gamma.12} that
\begin{equation}\label{eq.Gamma.21.part-aux}
    \biggl(\frac{\partial}{\partial x} - \frac{\partial}{\partial\zeta}\biggr) \Gamma_{12}(x,\zeta)
        = - \int_{0}^{\infty}\frac{\partial\Gamma_{11}}{\partial x} (x,t)F(x+t+\zeta)\,dt,
\end{equation}
so that
\[
    \frac{\partial}{\partial x}\Gamma_{11}(x,\zeta)
        = -w(x)F(x+t)  + \int_{0}^{\infty}\int_{0}^{\infty}
            \frac{\partial\Gamma_{11}} {\partial x}(x,t_{1})
            F(x+t_{1}+t_{2})F(x+t_{2}+\zeta)\,dt_{1}\,dt_{2}.
\]
Recalling now~\eqref{eq.Gamma}, we see that the function
\[
    N(x,\zeta)
        := \frac{\partial}{\partial x}\Gamma_{11}(x,\zeta)
            - w(x)\Gamma_{12}(x,\zeta)
\]
satisfies the relation
\[
    N(x,\zeta) - \int_{0}^{\infty}\int_{0}^{\infty}N(x,t_{1})
            F(x+t_{1}+t_{2})F(x+t_{2}+\zeta)\,dt_{1}\,dt_{2} =0.
\]
The analysis of the spectral properties of the operator~$T_{F}(x)$ given in
the proof of Proposition~\ref{pro.Gamma} implies that the above equation can
have only the trivial solution. Therefore~$N\equiv0$,
and~\eqref{eq.Gamma.11.partial} is established.

Using now~\eqref{eq.Gamma.11.partial} and \eqref{eq.Gamma.11} in
\eqref{eq.Gamma.21.part-aux}, we find that
\begin{align*}
    \biggl(\frac{\partial}{\partial x}
        - \frac{\partial}{\partial\zeta}\biggr)
            \Gamma_{12}(x,\zeta)
        & = -w(x) \int_{0}^{\infty}\Gamma_{12}(x,t)F(x+t+\zeta)\,dt\\
        & = w(x) \Gamma_{11}(x,\zeta)
\end{align*}
as claimed.
\end{proof}

Define now a vector-valued function $\bm{\psi}=(\psi_{1},\psi_{2})^{T}$ via
(cf.~\eqref{eq.Psi+.rep})
\[
    \bm{\psi}(x,k) := e^{ikx} \biggl[ \binom{1}{0}
        +  \int_{0}^{\infty}\bm{\Gamma}_{1}(x,\zeta)e^{2ik\zeta}\,d\zeta\biggr],
\]
where $\bm{\Gamma}_{1}:=(\Gamma_{11},\Gamma_{21})^{T}$ is the first column of the
matrix~$\Gamma$ solving~\eqref{eq.GLM}.

\begin{lemma}\label{lemma.AKNS}
The function~$\bm{\psi}$ solves the ZS-AKNS system~\eqref{eq.AKNS} with $z=k$ and
$u=w$.
\end{lemma}

\begin{proof}
We remark that since the entries~$\Gamma_{12}$ and~$\Gamma_{21}$ of $\Gamma$ coincide, we can freely interchange them as needed. Assume first that $F$ is as in Lemma~\ref{lemma.Gamma.partial}. Differentiation of the expression for~$\psi_{1}$ on account of the
relation~\eqref{eq.Gamma.11.partial} results in
\[
    \psi_{1}^{\prime}= ik\psi_{1} + w \psi_{2}.
\]
Similarly we find that
\begin{align*}
    \psi_{2}^{\prime}(x,k)  & + ik\psi_{2}(x,k) - w(x) \psi_{1}(x,k)\\
        & = -w(x) e^{ikx} + e^{ikx} \int_{0}^{\infty}
            \Bigl[2ik \Gamma_{12}+ \frac{\partial\Gamma_{12}}{\partial x}
            -w(x) \Gamma_{11}\Bigr](x,\zeta) e^{2ik\zeta}\,d\zeta\\
        & = -w(x) e^{ikx} + e^{ikx} \int_{0}^{\infty}
            \Bigl[2ik \Gamma_{12}(x,\zeta)e^{2ik\zeta} + \frac{\partial\Gamma_{12}}{\partial\zeta}(x,\zeta)
            e^{2ik\zeta}\Bigr]\,d\zeta\\
        & = -w(x) e^{ikx} + e^{ikx} \Gamma_{12}(x,\zeta)
            e^{2ik\zeta}\Bigr|_{\zeta=0}^{\infty} =0
\end{align*}
in view of~\eqref{eq.w} and the fact that $\Gamma_{12}(x,\zeta) \to0$ as
$\zeta\to\infty$ by virtue of~\eqref{eq.Gamma.12}. Therefore the lemma is proved for $F\in H^2(\mathbb{R})$.

Assume now that a real-valued function~$F\in X$ corresponds to a generic~$r\in\mathcal{R}_{>0}$. We approximate $F$ by a sequence of $F_n\in H^2(\mathbb{R})$ as follows. Let $\varphi$ be a nonnegative function in~$C_0^\infty(\mathbb{R})$ with $\int\varphi\,dx=1$ and set $\varphi_n(x):=n\varphi(nx)$ and $F_n:= F \ast \varphi_n$ for $n\in\mathbb{N}$, where~$\ast$ denotes the convolution. Then, clearly, $F_n$ belongs to~$H^2(\mathbb{R})$ and converges to $F$ in~$X$ as $n\to\infty$. Moreover, the Fourier transform $r_n:=\widehat{F}_n$ of $F_n$ is equal to $r \widehat{\varphi}_n$ and, since $|\widehat \varphi_n|\le1$, we find that $r_n$ belongs to~$\mathcal{R}$.

To show that $r_n \in \mathcal{R}_{>0}$, we set
\(
        \psi(k) := \bigl(1-|\widehat\varphi(k)|^2\bigr)/k^2
\)
and observe that
\[
        \psi_n(k) := \frac{1-|\widehat\varphi_n(k)|^2}{k^2}
            = \frac{\psi(k/n)}{n^2}.
\]
Therefore if we prove that $\psi \in \widehat X$, then the relation
\[
    \widetilde{r}_n(k) = \frac{1 - |r_n(k)|^2}{k^2} = \widetilde r(k)
            + |r(k)|^2 \psi_n(k)
\]
and the inequality $\psi_n(0)\ge0$ will imply that $\widetilde r_n \in \widehat{X}$ and $\widetilde r_n(0)>0$, i.e., that $r_n\in\mathcal{R}_{>0}$.

We observe first that $\widehat\varphi$ is of the Schwartz class and thus the same is true of~$|\widehat\varphi(k)|^2=\widehat\varphi(k) \widehat\varphi(-k)$. Next, the function $1-|\hat\varphi|^2$ has zero of order 2 at $k=0$ and thus $\psi$ belongs to~$C^\infty(\mathbb{R})$. The behaviour of~$\psi$ at infinity shows that it belongs to the Sobolev space~$H^1(\mathbb{R})$. Therefore~$\widehat \psi \in L^2(\mathbb{R})$ and, moreover,
\[
     \int_{\mathbb{R}} |\widehat \psi(k)|^2(1+k^2) dk <\infty,
\]
which by the Cauchy--Schwarz inequality yields $\widehat\psi\in L^1(\mathbb{R})$; hence $\psi\in \widehat X$ as required. (We note in passing that in fact we have the convergence of $r_n$ to $r$ in the topology of~$\mathcal{R}_{>0}$, although this fact is not needed.)

Now, for every $n\in\mathbb{N}$, we denote by $w_n$ the function of~\eqref{eq.w} corresponding to the solution $\bm{\Gamma}_{n}$ of~\eqref{eq.GLM} with $\Omega$ constructed for~$F_n$ instead of~$F$ and by $Q_n$ the matrix of~\eqref{eq.sigmaQ} with $u=w_n$. By Proposition~\ref{pro.Gamma}, the matrix-valued functions $Q_n$ converge to $Q$ in $X_c^+$ componentwise as $n\to\infty$, while the functions $\bm{\psi}_n(\cdot,k)$ converge to~$\bm{\psi}$ in the uniform topology on $(c,\infty)$, for every $c,\, k\in\mathbb{R}$. It follows from~\eqref{eq.AKNS} that, on every compact $x$-interval~$\Delta$, the functions $\frac{d}{dx}\bm{\psi}_n$ converge in the topology of $L^1(\Delta)\cap L^2(\Delta)$. It follows that, for every fixed $k\in\mathbb{R}$, $\bm{\psi}_n$ converge in $W^{1,1}(\Delta)\cap W^{1,2}(\Delta)$ to $\bm{\psi}$ and that $\bm{\psi}$ satisfies the ZS-AKNS system~\eqref{eq.AKNS}. The proof is complete.
\end{proof}

Set $K:=\Gamma_{11}+\Gamma_{21}$; then $K$ satisfies the
Gelfand--Levitan--Marchenko equation
\begin{equation}\label{eq.GLM.K}
    K(x,\zeta)+F(x+\zeta)+\int_{0}^{\infty}K(x,t)F(x+t+\zeta)\,dt = 0
\end{equation}
for $\zeta>0$, and the function $f:= \psi_{1} + \psi_{2}$ has the representation
\[
    f(x,k) := e^{ikx}\left(  1+\int_{0}^{\infty}K(x,\zeta)e^{2ik\zeta}
            \,d\zeta\right).
\]
By the preceding lemma the vector-valued function $\bm{\psi}$ solves the ZS-AKNS system~\eqref{eq.AKNS} with $u=w$; therefore, as explained in Subsection~\ref{subsec.jost}, the above function~$f$ solves the Schr\"{o}dinger equation $-y^{\prime\prime}+ q y = k^{2} y$
with $q=w^{\prime}+w^{2}$. In other words, we arrive at the following conclusion.

\begin{proposition}\label{pro.f+}
The function $f$ is the ``right'' Jost solution for the half-line
Schr\"{o}dinger operators with the ``right'' Riccati representative~$w$.
\end{proposition}

Similar analysis of the kernel~$\Gamma^{\#}$ yields a kernel~$K^{\#}$
satisfying the Gelfand--Levitan--Marchenko equation
\begin{equation}\label{eq.GLM.Ksharp}
    K^{\#}(x,\zeta)
        + F^{\#}(x+\zeta)+\int_{-\infty}^{0}K(x,t)F^{\#}(x+t+\zeta)\,dt=0
\end{equation}
for $\zeta<0$ and the function
\[
    f^{\#}(x,k)
        := e^{-ikx} \left(  1+\int_{-\infty}^{0}K^{\#}(x,\zeta)e^{-2ik\zeta}
            \,d\zeta\right)
\]
solving the Schr\"{o}dinger equation $-y''+ q^{\#} y = k^{2} y$
with $q^{\#}=(w^{\#})^{\prime}+(w^{\#})^{2}$.

As with $f$, we have

\begin{proposition}\label{pro.f-}
The function $f^{\#}$ is the ``left'' Jost solution for the half-line
Schr\"{o}dinger operators with the ``left'' Riccati representative~$w^{\#}$.
\end{proposition}

\subsection{Consistency of the reconstruction}
\label{subsec.consistent}

Now we are in a position to justify the reconstruction procedure as suggested in Subsection~\ref{subsec.GLM} and prove the following:

\begin{theorem}
\label{thm.right} The Schr\"{o}dinger operator with potential~$q_0$ corresponding to
the triple
 $\bigl(  w\bigl|_{\mathbb{R^+}},w^{\#}\bigl|_{\mathbb{R^-}},(w^{\#}-w)(0)\bigr)  $
has the \textquotedblleft right\textquotedblright\ reflection coefficient $r$.
\end{theorem}

Given $r\in\mathcal{R}_{>0}$, we construct the function~$\widetilde t$ via~\eqref{eq.tildet.rep} and set $t(k) = k \widetilde t(k)/(k+i)$. In what follows  we will denote by $a$ and $b$ the functions $1/t$ and $ar^{\#}$. Our aim is to show that $a$ and $b$ so defined coincide with the corresponding
coefficients $a_0$ and~$b_0$ for the potential~$q_0$
reconstructed from the data (\ref{eq.riccati.data}).

The explicit construction formula~\eqref{eq.tildet.rep} yields the following properties of the function~$a$:
\begin{itemize}
\item[(1)] $a$ is analytic in $\mathbb{C}^{+}$ and continuous in $\overline{\mathbb{C}^+}\backslash \{0\}$;
\item[(2)] $a(k)\rightarrow 1$  as $|k| \rightarrow\infty$ in $\overline{\mathbb{C}^{+}}$ and $\lim_{k\to0}2ik\,a(k)=\theta\ne0$, for some $\theta$ uniquely determined by~$r$.
\end{itemize}
Also, we see from~\eqref{eq.tildet.rep} that $t$ satisfies the symmetry relation~$t(-k) = \overline{t(k)}$ and that
\[
    t(-k) = \frac{k}{k-i} \exp \left\{- \mathcal{C}_-
        \log \Bigl[ \bigl(1-\left\vert r(s)\right\vert ^{2}\bigr)\frac{s^2+1}{s^2} \Bigr]\right\}
\]
with $\mathcal{C}_-$ the Riesz projector of~\eqref{eq.Cauchy}. Recalling the relation~$\mathcal{C}_+ - \mathcal{C}_- = I$, we conclude that
\[
    |t(k)|^2 = t(k)t(-k) = 1-|r(k)|^2,
\]
i.e.,
\begin{equation}\label{eq.ab1.recon}
    |a(k)|^2 - |b(k)|^2 = 1.
\end{equation}

Introduce the functions
\begin{align*}
    m(x,k) &:= f(x,k) e^{-ikx},\\
    m^{\#}(x,k) &:= f^{\#}(x,k) e^{ikx};
\end{align*}
then we have (cf.~Lemma~\ref{lemma.hardy}):

\begin{lemma}\label{lemma.r.hardy}
For every fixed $x\in\mathbb{R}$, the functions
\begin{equation}
    m(x,-k)+e^{2ikx}r(k)m(x,k) \label{eq.3}%
\end{equation}
and%
\begin{equation}
    m^{\#}(x,-k)+e^{-2ikx}r^{\#}(k)m^{\#}(x,k) \label{eq.4}%
\end{equation}
admit analytic continuations into $\mathbb{C}^{+}$ as elements of
$H_{+}^{2}(\mathbb{R)}$, which are bounded and continuous in~$\overline{\mathbb{C}^+}$.
\end{lemma}

\begin{proof}
To derive the properties of~(\ref{eq.3}), consider the function
\[
    G(x,y):= K(x,y)+F(x+y)+\int_{0}^{\infty}K(x,s)F(x+s+y)\,ds
\]
on the whole line, where we assume $K(x,y)$ to be continued by zero for $y<0$. The function~$G(x,\,\cdot\,)$ belongs to $X$ and vanishes for $y>0$ by (\ref{eq.GLM.K}). Therefore the inverse Fourier transform $\mathcal{F}^{-1}G$ of~$G$ is an analytic function in $\mathbb{C}^+$ that belongs to $H_{+}^{2}(\mathbb{R)}$, is continuous up to the boundary and bounded in the closed upper-half plane~$\overline{\mathbb{C}^+}$. It remains to observe that $\pi\mathcal{F}^{-1}G$ coincides with~$m(x,-k)+e^{2ikx}r(k)m(x,k)-1$ on the real line.

One obtains properties of~(\ref{eq.4}) similarly by taking the Fourier transform of~(\ref{eq.GLM.Ksharp}).
\end{proof}

Next, we prove:

\begin{lemma}\label{lemma.fsharp}
The following relations hold:%
\begin{align}
    f^{\#}(x,k)  &  =a(k)f(x,-k)-b(-k)f(x,k),\label{eq.5}\\
    f(x,k)  &  =a(k)f^{\#}(x,-k)+b(k)f^{\#}(x,k). \label{eq.6}%
\end{align}
\end{lemma}

\begin{proof}
Denote by $g^{\#}(x,k)$ (resp. $g(x,k)$) the right-hand side of (\ref{eq.5})
(resp. (\ref{eq.6})). The relation
\[
    g^{\#}(x,k) = e^{-ikx}a(k) \bigl[m(x,-k)+ e^{2ikx}r(k) m(x,k)\bigr]
\]
in view of Lemma~\ref{lemma.r.hardy} shows that $e^{ikx}g^{\#}(x,k)$ admits analytic continuation to the open upper-half plane~$\mathbb{C}^+$ that is continuous on~$\overline{\mathbb{C}^+} \setminus\{0\}$ and bounded there outside every neighbourhood of the origin. Since the expression in the square brackets above vanishes at $k=0$, we conclude that $g^{\#}(x,k)= o (1/k)$ as $k \to0$ within $\overline{\mathbb{C}^+} \setminus\{0\}$. Similar arguments show that the function~$e^{-ikx}g(x,k)$ enjoys the same analyticity and continuity properties for $k\in\mathbb{C}^+$.

Solving the system%
\begin{align*}
    g^{\#}(x,k)  &  =a(k)f(x,-k)-b(-k)f(x,k),\\
    g^{\#}(x,-k)  &  =a(-k)f(x,k)-b(k)f(x,-k)
\end{align*}
for $f(x,k)$ on account of~(\ref{eq.ab1.recon}) gives
\begin{equation}
    f(x,k)=a(k)g^{\#}(x,-k)+b(k)g^{\#}(x,k). \label{eq.theid}%
\end{equation}
From (\ref{eq.theid}) and the definition of $g$, we compute%
\begin{equation}
    \frac{f(x,k)f^{\#}(x,k)-g(x,k)g^{\#}(x,k)}{a(k)}
        =g^{\#}(x,-k)f^{\#}(x,k)-f^{\#}(x,-k)g^{\#}(x,k). \label{eq.7}%
\end{equation}
The right-hand side of (\ref{eq.7}) is an odd function of $k\in\mathbb{R}$,
while the left-hand side has an analytic extension to the upper complex
half-plane that is continuous up to~$\mathbb{R}\setminus\{0\}$. Hence we can extend the left-hand side of~\eqref{eq.7} to an analytic function~$h$ on $\mathbb{C}\setminus\{0\}$. Since the right-hand side of~\eqref{eq.7} is~$o(1/k)$ as $k\to0$, $k=0$ is a removable singularity of $h$ and thus $h$ is an entire function. Note that $h$ is bounded in $\mathbb{C}^+$ because such are the functions $f(x,\,\cdot\,)f^{\#}(x,\,\cdot\,)$, $g(x,\,\cdot\,)g^{\#}(x,\,\cdot\,)$, and $1/a$. Since $h$ is odd, it is bounded in~$\mathbb{C}$ and thus a constant, which must be zero.

We thus conclude that%
\[
    g^{\#}(x,-k)f^{\#}(x,k)=f^{\#}(x,-k)g^{\#}(x,k),
\]
and since the Jost solution $f^{\#}(x,k)$ never vanishes for real~$x$ and real nonzero~$k$, we get
\begin{equation}
    \frac{g^{\#}(x,-k)}{f^{\#}(x,-k)}=\frac{g^{\#}(x,k)}{f^{\#}(x,k)}.
\label{eq.10}%
\end{equation}
The left-hand side defines a function analytic and bounded in $\mathbb{C}^{-}$, while the right-hand side defines a function analytic and bounded in $\mathbb{C}^{+}$. Thus both sides give a function that is analytic in $\mathbb{C}\setminus\mathbb{R}$ and continuous up to $\mathbb{R}\setminus\{0\}$. Arguing as above, we conclude that this function is analytic in~$\mathbb{C}\setminus\{0\}$ and has a removable singularity at $k=0$. We thus get a bounded entire function, which must be constant. Since both $g^{\#}(x,k)$ and $f^{\#}(x,k)$ tend to~$1$ when $k$ tends to~$\infty$ along the real line, this constant is~$1$, and thus
\[
    g^{\#}(x,k)=f^{\#}(x,k)
\]
as claimed, so (\ref{eq.5}) holds.

A similar proof shows that $g(x,k)=f(x,k)$ so that (\ref{eq.6}) holds.
\end{proof}

Set now $q:= w' + w^2$ and $q^{\#}:= (w^{\#})' + (w^{\#})^2$; then $q$ and $q^{\#}$ are distributions in $H^{-1}_{\mathrm{loc}}(\mathbb{R})$. The crucial result is given by the following lemma.

\begin{lemma}\label{lemma.qsharp}
$q$ and $q^{\#}$ coincide as distributions in~$H^{-1}_{\mathrm{loc}}(\mathbb{R})$.
\end{lemma}

\begin{proof}
For every real nonzero~$k$, the functions $f(\,\cdot\,,k)$ and $f(\,\cdot\,,-k)$ are linearly independent solutions of the equation
\[
    -y'' + q y = k^2 y,
\]
while
$f^{\#}(\,\cdot\,,k)$ and $f^{\#}(\,\cdot\,,-k)$ are linearly independent solutions of the equation
\begin{equation}\label{eq.qsharp}
    -y'' + q^{\#} y = k^2 y.
\end{equation}
Lemma~\ref{lemma.fsharp} shows that the function $f(\,\cdot\,,k)$ also solves equation~\eqref{eq.qsharp} and thus we get the equality
\[
    (q-q^{\#}) f(\,\cdot\,,k) =0
\]
in the distributional sense for all real nonzero~$k$. We recall that, in virtue of Lemma~\ref{lemma.jost}, the Jost solution~$f(\,\cdot\,,0)$ is everywhere positive on~$\mathbb{R}$. In view of the analytic dependence on~$k$, for every $x_0\in\mathbb{R}$ there exists a real nonzero $k_0$ such that $f(x_0,k_0)>0$. Therefore $f(x,k_0)>0$ for all $x$ in some neighbourhood of~$x_0$, whence $q$ and $q^{\#}$ coincide as distributions in~$H^{-1}_{\mathrm{loc}}(\mathbb{R})$ in this neighbourhood. Since $x_0$ was arbitrary, we claim follows.
\end{proof}

The above lemma implies that~$w$ and~$w^{\#}$ are right and left Riccati representatives of a distribution $q_0 \in \mathcal{Q}_{>0}$ that in view of Lemma~\ref{lemma.r1} can be associated to the triple
$\left(  w,w^{\#},(w^{\#}-w)(0)\right)$. By Propositions~\ref{pro.f+} and \ref{pro.f-}, the Jost solutions $f_{\pm}(\,\cdot\,,k)$ for the potential~$q_0$ satisfy for all $k\in\mathbb{R}$ the equalities
\[
    f_{+}(\,\cdot\,,k)    =f(\,\cdot\,,k),
        \qquad
    f_{-}(\,\cdot\,,k)    =f^{\#}(\,\cdot\,,k).
\]
Lemma~\ref{lemma.fsharp} gives, for all real $x$ and $k$, the following relations:
\begin{align*}
    f_-(x,k)   &  =a(k)f_+(x,-k) - b(-k)f_+(x,k),\\
    f_+(x,k)  &  =a(k)f_-(x,-k)+b(k)f_-(x,k).
\end{align*}
It now follows that ${q}_0$ has reflection coefficients $r_{+}=r$ and $r_{-}=r^{\#}$ and thus is indeed the potential in~$\mathcal{Q}_{>0}$ looked for. This completes the reconstruction procedure and proves Theorem~\ref{thm.right}.

\medskip

\emph{Acknowledgements.}
This material is based upon work supported by the National Science Foundation under Grant DMS-0710477 (RH and PP) and by the Deutsche Forschungsgemeinschaft under project 436 UKR 113/84 (RH and YM). RH acknowledges support from the College of Arts and Sciences at the University of Kentucky and thanks the Department of Mathematics at the University of Kentucky for hospitality during his stay there. RH, YM, and PP thank the Institut f\"ur angewandte Mathematik der Universit\"at Bonn for hospitality during part of the time that this work was done. PP thanks Percy Deift for helpful conversations and SFB 611 for financial support of his research visit to Universit\"at Bonn. The authors thank Iryna Egorova for bringing to their attention the paper~\cite{Guseinov:1985}.

\end{document}